\documentclass[12pt]{article}%
\usepackage[colorlinks,bookmarks=true]{hyperref}
\usepackage{amsmath}
\usepackage{graphicx}
\usepackage{amsfonts}
\usepackage{amssymb}%
\providecommand{\U}[1]{\protect\rule{.1in}{.1in}}
\newtheorem{theorem}{Theorem} [section]

\newtheorem{corollary}{Corollary}[section]

\newtheorem{lemma}{Lemma}[section]

\newtheorem{proposition}{Proposition} [section]
\newtheorem{remark}{Remark} [section]

\newenvironment{proof}[1][Proof]{\textbf{#1.} }{\ \rule{1em}{1em}}
\numberwithin{equation}{section}

\begin{document}

\title{Metastability of Kolmogorov flows and inviscid damping of shear flows}
\author{Zhiwu Lin\\School of Mathematics\\Georgia Institute of Technology\\Atlanta, GA 30332, USA
\and Ming Xu\\Department of Mathematics\\JiNan University\\Guangzhou, 510632, China}
\date{}
\maketitle

\begin{abstract}
First, we consider Kolmogorov flow (a shear flow with a sinusoidal velocity
profile) for 2D Navier-Stokes equation on a torus. Such flows, also called bar
states, have been numerically observed as one type of metastable states in the
study of 2D turbulence. For both rectangular and square tori, we prove that
the non-shear part of perturbations near Kolmogorov flow decays in a time
scale much shorter than the viscous time scale. The results are obtained for
both the linearized NS equations with any initial vorticity in $L^{2}$, and
the nonlinear NS equation with initial L%
${{}^2}$
norm of vorticity of the size of viscosity. In the proof, we use the
Hamiltonian structure of the linearized Euler equation and RAGE theorem to
control the low frequency part of the perturbation. Second, we consider two
classes of shear flows for which a sharp stability criterion is known. We show
the inviscid damping in a time average sense for non-shear perturbations with
initial vorticity in $L^{2}$. For the unstable case, the inviscid damping is
proved on the center space. Our proof again uses the Hamiltonian structure of
the linearized Euler equation and an instability index theory recently
developed by Lin and Zeng for Hamiltonian PDEs.

\end{abstract}

\section{Introduction}

Consider 2D Navier-Stokes (NS) equation
\begin{equation}
\partial_{t}U+U\cdot\bigtriangledown U-\nu\bigtriangleup U=-\bigtriangledown
P,\text{ } \label{eqn-NS}%
\end{equation}
on a torus
\[
\mathbb{T}_{\alpha}=\left\{  0<y<2\pi,0<x<\frac{2\pi}{\alpha}\right\}
,\ \alpha>0,
\]
with the incompressible condition $\nabla\cdot U=0$, where $U=\left(
u,v\right)  $ is the fluid velocity and $\nu>0$ is the viscosity. More
precisely, we impose the periodic boundary conditions
\[
U\left(  0,y,t\right)  =U\left(  2\pi/\alpha,y,t\right)  ,\ U\left(
x,0,t\right)  =U\left(  x,2\pi,t\right)  \ .
\]
The vorticity form of NS equation (\ref{eqn-NS}) is
\begin{equation}
\omega_{t}+u\omega_{x}+v\omega_{y}-\nu\bigtriangleup\omega=0,\ \ \omega
=v_{x}-u_{y}. \label{eqn-NS-vorticity}%
\end{equation}
It is convenient to introduce the stream function $\psi$ such that
$\omega=-\bigtriangleup\psi$ and $U=\nabla^{\perp}\psi=\left(  \psi_{y}%
,-\psi_{x}\right)  $.

In the numerical and experimental study of $2$D turbulence, it has often been
observed (\cite{mongomery-et-91} \cite{yin-et-final state}
\cite{bouchet-simmonet09}) that the solutions to the two-dimensional
Navier-Stokes (NS) equations with small viscosity rapidly approach certain
long-lived coherent structures. Evidences also suggested that these
quasi-stationary, or metastable solutions are closely related to stationary
solutions of the inviscid Euler equations%
\[
\omega_{t}+u\omega_{x}+v\omega_{y}=0,\ \ \omega=v_{x}-u_{y}.
\]
Since there is no forcing in (\ref{eqn-NS-vorticity}), when $t\rightarrow
\infty$, $\left\Vert \omega\left(  t\right)  \right\Vert _{L^{2}}\rightarrow0$
in the viscous time scale $O\left(  \frac{1}{\nu}\right)  $, where
$\omega\left(  t\right)  $ is the solution of (\ref{eqn-NS-vorticity}) with
initial data$\ \omega\left(  0\right)  \in L^{2}$. We are interested in the
dynamics of (\ref{eqn-NS-vorticity}), particularly, the appearance and
persistence of coherent states in the intermediate time scale$\ \left(
0,T\right)  $, where $1\ll T\ll O\left(  \frac{1}{\nu}\right)  $.\ The first
step is to prove that nearby solutions converge rapidly to these coherent
states in a time scale $T\ll O\left(  \frac{1}{\nu}\right)  $. Such
metastability problem is also called enhanced damping in the literature. Among
the candidates of Euler steady solutions to explain the coherent structures,
some authors (e.g. \cite{bouchet-simmonet09} \cite{mongomery-et-91}
\cite{yin-et-final state}) suggested that certain maximal entrophy solutions
of the inviscid Euler equation are the most probable quasi-stationary that one
would observe. The simplest of such maximal entrophy solutions is the
Kolmogorov flow (also called bar states in \cite{yin-et-final state}), that
is, $u_{0}=\left(  \sin y,0\right)  $ or $\left(  \cos y,0\right)  $. The
solution to (NS) with initial data $u_{0}$ is $u^{\nu}\left(  t,y\right)
=e^{-\nu t}\left(  \sin y,0\right)  $. The linearized (NS) equation near
$u^{\nu}\ $is
\begin{equation}
\partial_{t}\omega=\nu\Delta\omega-e^{-\nu t}\left[  \sin y\partial_{x}\left(
1+\Delta^{-1}\right)  \right]  \omega=L\left(  t\right)  \omega,
\label{eqn-bar-LNS}%
\end{equation}
where $\nu$ is the viscosity and $\omega$ is the vorticity perturbation. In
\cite{beck-wayne}, Beck and Wayne studied the following approximation of the
linearized problem
\begin{equation}
\partial_{t}\omega=\nu\Delta\omega-e^{-\nu t}\sin y\partial_{x}\omega
=\tilde{L}\left(  t\right)  \omega, \label{LNS-appro}%
\end{equation}
by dropping the nonlocal term $e^{-\nu t}\sin y\partial_{x}\Delta^{-1}\omega$
in (\ref{eqn-bar-LNS}). Define the following weighed $H^{1}\ $space for
non-shear vorticity functions
\begin{align}
Z  &  ={\Huge \{}\sum_{k\neq0}\omega=\omega_{k}\left(  y\right)  e^{ikx}\in
L^{2},\label{space-Z-B-W}\\
\left\Vert \omega\right\Vert _{Y}^{2}  &  :=\sum_{k\neq0}\left[  \left\Vert
\omega_{k}\right\Vert _{2}^{2}+\sqrt{\frac{\nu}{\left\vert k\right\vert }%
}\left\Vert \partial_{y}\omega_{k}\right\Vert _{2}^{2}+\frac{1}{\sqrt{\nu
}\left\vert k\right\vert ^{\frac{3}{2}}}\left\Vert C^{k}\omega_{k}\right\Vert
_{2}^{2}\right]  <\infty{\Huge \},}\nonumber
\end{align}
where $C^{k}\omega_{k}=-ike^{\nu t}\left(  \cos y\right)  \omega_{k}$. It was
proved in \cite{beck-wayne} that: For any $\tau>0$ and $T\in\left[
0,\frac{\tau}{\nu}\right]  $, there exist constants $K,\ M$ such that: if
$\nu$ is small enough, then the solution to (\ref{LNS-appro}) with initial
data $\omega\left(  0\right)  \in Z$ satisfies the estimate
\[
\left\Vert \omega\left(  t\right)  \right\Vert _{Z}^{2}\leq Ke^{-M\sqrt{\nu}%
t}\left\Vert \omega\left(  0\right)  \right\Vert _{Z}^{2},\ t\in\left[
0,T\right]  \text{. }%
\]
The proof used Villani's hypocoercivity method (\cite{villani-hypocoecivity}).
For the full linearized NS equation, numerical evidences in \cite{beck-wayne}
suggest the same decay rate $O\left(  e^{-\sqrt{\nu}t}\right)  $.

In this paper, we study the full linearized equation (\ref{eqn-bar-LNS}) and
the nonlinear equation (\ref{eqn-NS}) on a torus
\[
\mathbb{T}_{\alpha}=\left\{  0<y<2\pi,0<x<\frac{2\pi}{\alpha}\right\}  ,
\]
with $\alpha\geq1$, which is the sharp stability condition of Kolmogorov flows
for the 2D Euler equation (see Lemma 4.1 of \cite{lin-zeng-Euler-mfld}). Our
first result is about the enhanced damping for the linearized NS\ equation.

\begin{theorem}
\label{thm-linearized}Consider the linearized NS equation (\ref{eqn-bar-LNS})
on $\mathbb{T}_{\alpha}\ $with $\alpha\geq1$. Define the non-shear vorticity
space
\begin{equation}
X=\left\{  \omega\in L^{2}|\ \omega=\sum_{0\neq k\in\mathbf{Z}}\omega
_{k}\left(  y\right)  e^{ik\alpha x}\right\}  . \label{definition-X}%
\end{equation}

i) (Rectangular torus) Consider $\alpha>1$. For any $\tau>0$ and $\delta>0$,
if $\nu$ is small enough, then the solution $\omega\left(  t\right)  \ $to
(\ref{eqn-bar-LNS}) with non-shear initial data $\omega\left(  0\right)  \in
X$ satisfies $\left\Vert \omega\left(  \frac{\tau}{\nu}\right)  \right\Vert
_{L^{2}}<\delta\left\Vert \omega\left(  0\right)  \right\Vert _{L^{2}}.$

ii) (Square torus) Consider $\alpha=1$. Let $P_{1}$ be the orthogonal
projection from the non-shear space
\[
X=\left\{  \omega\in L^{2}|\ \omega=\sum_{k\neq0}\omega_{k}\left(  y\right)
e^{ikx}\right\}
\]
to the space $W_{1}\ $spanned by $\left\{  \cos x,\sin x\right\}  $. For any
$\tau>0$ and $\delta>0$, if $\nu$ is small enough, then the solution to
(\ref{eqn-bar-LNS}) with initial data $\omega\left(  0\right)  \in X$
satisfies
\begin{equation}
\left\Vert \left(  I-P_{1}\right)  \omega\left(  \frac{\tau}{\nu}\right)
\right\Vert _{L^{2}}<\delta\left\Vert \left(  I-P_{1}\right)  \omega\left(
0\right)  \right\Vert _{L^{2}}. \label{enhanced damping torus-linear}%
\end{equation}

\end{theorem}

Since $\tau$ can be arbitrarily small, above result implies a much enhanced
decay in the time scale $O\left(  \frac{\tau}{\nu}\right)  $ compared with the
viscous time scale $O\left(  \frac{1}{\nu}\right)  $. For shear initial data
$\omega\left(  0\right)  =\omega_{0}\left(  y\right)  ,\ $the linearized NS
equation (\ref{eqn-bar-LNS}) is reduced to the heat equation $\partial
_{t}\omega=\nu\partial_{yy}\omega$ and there is no enhanced decay. On the
square torus, there are two additional kernels $\left\{  \cos x,\sin
x\right\}  $ of the operator $1+\Delta^{-1}$, which correspond to exact
solutions $e^{-\nu t}\left\{  \cos x,\sin x\right\}  $ of the Navier-Stokes
equations. These two additional kernels (so called anomalous modes in
\cite{beck-wayne}) need to be removed for the enhanced damping to hold true.

For the nonlinear Navier-Stokes equation, we have the following result.

\begin{theorem}
\label{thm-nonlinear}Consider the nonlinear NS equation (\ref{eqn-NS}) on
$\mathbb{T}_{\alpha}\ $with $\alpha\geq1$. Denote $P_{2}$ to be the projection
of $L^{2}\left(  \mathbb{T}_{\alpha}\right)  $ to the subspace $W_{2}%
=span\left\{  \cos y,\sin y\right\}  $.

i) (Rectangular torus) Suppose $\alpha>1$. There exist $d>0$, such that: for
any $\tau>0$ and $\delta>0$, if $\nu$ is small enough, then any solution to
(\ref{eqn-NS}) with initial data $\omega\left(  0\right)  \in L^{2}$ and
\begin{equation}
\left\Vert \left(  I-P_{2}\right)  \omega\left(  0\right)  \right\Vert
_{L^{2}}\leq d\nu, \label{initial-data-small}%
\end{equation}
satisfies
\[
\left\Vert P_{\neq0}\omega\left(  \frac{\tau}{\nu}\right)  \right\Vert
_{L^{2}}<\delta\left\Vert P_{\neq0}\omega\left(  0\right)  \right\Vert
_{L^{2}}.
\]
Here, $P_{\neq0}$ is the projection of $L^{2}$ to the non-shear space $X$,
that is,
\[
P_{\neq0}\omega=\omega-\frac{\alpha}{2\pi}\int_{0}^{\frac{2\pi}{\alpha}}\omega
dx.
\]

ii) (Square torus) Suppose $\alpha=1$. There exist $d>0$, such that: for any
$\tau>0$ and $\delta>0$, if $\nu$ is small enough, then any solution to
(\ref{eqn-NS}) with initial data $\omega\left(  0\right)  \in L^{2}$ and
$\left\Vert \left(  I-P_{2}\right)  \omega\left(  0\right)  \right\Vert
_{L^{2}}\leq d\nu$, satisfies
\begin{equation}
\inf_{0\leq t\leq\frac{\tau}{\nu}}\left\Vert \left(  1-P_{1}\right)  P_{\neq
0}\omega\left(  t\right)  \right\Vert _{L^{2}}<\delta\left\Vert P_{\neq
0}\omega\left(  0\right)  \right\Vert _{L^{2}}. \label{enhanced damping torus}%
\end{equation}
Here, the projections $P_{1},P_{\neq0}$ are defined as above.
\end{theorem}

In the above theorem, the metastability of Kolmogorov flow is proved for
perturbations of the size $\nu$. More precisely, it is shown that the
non-shear part of the perturbation is reduced to a factor $\delta$ of the
initial size before the time scale$\frac{\tau}{\nu}$, which is much smaller
than the viscous time scale $\frac{1}{\nu}$. Moreover, by choosing the initial
data to be smaller, we can ensure that $\left\Vert \left(  I-P_{2}\right)
\omega\left(  t\right)  \right\Vert _{L^{2}}\leq d\nu$ for all $t>0$, thus we
can repeatedly use Theorem \ref{thm-nonlinear} to get the rapid decay of the
non-shear part before the viscous time scale. We refer to Remark
\ref{remark-repeat-decay} for more details.

Next, we discuss some key ideas in the proof of Theorems \ref{thm-linearized}
and \ref{thm-nonlinear}. Our work is partly motivated by the work of
Constantin et al. \cite{kiselev-annals-mixing} for the linear reaction
diffusion equation
\[
\partial_{t}\phi+v_{0}\cdot\nabla\phi-\nu\Delta\phi=0,
\]
with an incompressible flow $v_{0}\left(  x\right)  $. In
\cite{kiselev-annals-mixing}, the enhanced damping in the sense of Theorem
\ref{thm-linearized} is proved under the assumption that the operator
$v_{0}\cdot\nabla$ has no non-constant eigenfunction in $H^{1}$. Their proof
is to consider the high and low frequency parts of the solution $\phi\left(
t\right)  \ $separately. For the high frequency part (i.e. $\left\Vert
\nabla\phi\right\Vert _{L^{2}}\thickapprox N\left\Vert \phi\right\Vert
_{L^{2}}$ for $N$ large), the enhanced damping is ensured by the energy
dissipation law%
\begin{equation}
\partial_{t}\left\Vert \phi\right\Vert _{L^{2}}^{2}=-\nu\left\Vert \nabla
\phi\right\Vert _{L^{2}}^{2}.\label{dissipation-law-heat}%
\end{equation}
The lower frequency part is shown to converge to zero in the time average
sense, by using the following RAGE Theorem for the unitary group $e^{itL}$
with the self-adjoint generator $L=iv_{0}\cdot\nabla$.

\textbf{Theorem (RAGE)} \cite{CFKS} Let $L$ be a self-adjoint operator in a
Hilbert space $H$, $P_{c}$ is the projection to the continuous spectrum space
of $L$ and $B$ is any compact operator, then
\[
\frac{1}{T}\int_{0}^{T}\left\Vert Be^{itL}P_{c}\psi\right\Vert _{H}%
^{2}dt\rightarrow0,\ \text{when }T\rightarrow\infty.
\]

In the proof of enhanced damping, $B$ is taken to be the projection to the low
frequency modes. Then the RAGE Theorem implies that the low frequency modes
decay in the time average sense.

To apply these ideas to prove the enhanced damping for the linearized
Navier-Stokes equation (\ref{eqn-bar-LNS}), there are a few difficulties to be
overcome. First, for the equation (\ref{eqn-bar-LNS}), there is no obvious
dissipation law as (\ref{dissipation-law-heat}). We derive the following
identity
\begin{equation}
\frac{d}{dt}\int_{\mathbb{T}_{\alpha}}(|\omega|^{2}-|\bigtriangledown\psi
|^{2})dxdy=-2\nu\int_{\mathbb{T}_{\alpha}}(|\bigtriangledown\omega
|^{2}-|\omega|^{2})dxdy, \label{dissipation-law-NS}%
\end{equation}
where $\psi=\left(  -\Delta\right)  ^{-1}\omega$ is the stream function. When
$\alpha>1$, the quadratic forms on both sides of (\ref{dissipation-law-NS})
are positive definite for non-shear vorticity (i.e. $\omega\in X$). When
$\alpha=1$, the positivity is still true in the space $X_{1}=\left(
I-P_{1}\right)  X_{0}$. This provides a substitute of
(\ref{dissipation-law-heat}).

Second, even if we ignore the factor $e^{-\nu t}$ in (\ref{eqn-bar-LNS}), the
linearized Euler operator $A=-\sin y\partial_{x}\left(  1+\Delta^{-1}\right)
$ is not anti-self-adjoint and the RAGE\ theorem cannot be applied directly to
$e^{tA}$. An important observation is that $A$ can be written in the
Hamiltonian form $A=JL$, where
\begin{equation}
J=-\sin y\partial_{x},\ \ \ L=1+\Delta^{-1}\label{defn-J-L}%
\end{equation}
are anti-selfadjoint and adjoint operators respectively in $L^{2}$. When
$\alpha>1$, since $L=1+\Delta^{-1}>0$ on the non-shear space $X$, we can
define a new inner product by $\left[  \cdot,\cdot\right]  =\left\langle
L\cdot,\cdot\right\rangle \ $on $X$, which is equivalent to the $L^{2}\ $inner
product. We observe that the operator $A$ is anti-selfadjoint in the space
$\left(  X,\left[  \cdot,\cdot\right]  \right)  $. Moreover, on the space
$X,\ $the operator $A$ can be shown to have no embedded eigenvalues in the
continuous spectra. Thus RAGE theorem can be applied to the semigroup $e^{tA}$
to show the decay of the low frequency part in the time average sense. The
linear enhanced damping for (\ref{eqn-bar-LNS}) then follows similarly as in
\cite{kiselev-annals-mixing}. For the square torus ($\alpha=1$), there are
additional anomalous modes $\left\{  \cos x,\sin x\right\}  $ lying in
$\ker\left(  1+\Delta^{-1}\right)  $. For any initial data $\omega\left(
0\right)  \in X$, we note that $\omega_{1}=\left(  I-P_{1}\right)  \omega$
satisfies the equation
\[
\partial_{t}\omega_{1}=\nu\Delta\omega_{1}-e^{-\nu t}\left[  \left(
I-P_{1}\right)  \sin y\partial_{x}\left(  1+\Delta^{-1}\right)  \right]
\omega_{1}.
\]
Let
\[
A_{1}=-\left(  I-P_{1}\right)  \sin y\partial_{x}\left(  1+\Delta^{-1}\right)
=\left(  I-P_{1}\right)  JL,
\]
then it can be checked that $A_{1}$ is anti-selfadjoint on the space
$X_{1}=\left(  I-P_{1}\right)  X$ in the inner product $\left[  \cdot
,\cdot\right]  =\left\langle L\cdot,\cdot\right\rangle $, where the positivity
of $L|_{X_{1}}$ is used. So applying the RAGE theorem to the semigroup
$e^{tA_{1}}$ on $X_{1}$, we can again show that the low frequency part of
$\left(  I-P_{1}\right)  e^{tJL}\omega\left(  0\right)  $ decays in the time
average sense.

For the nonlinear NS equation (\ref{eqn-NS-vorticity}), the evolution of the
shear and non-shear parts are strongly coupled. For an initial perturbation
$\omega\left(  0\right)  \ $of size $O\left(  \nu\right)  \ $in $L^{2}$, the
interaction terms are controllable and the nonlinear enhanced damping
(metastability) still holds true. On the square torus, the analysis for the
nonlinear problem is more involved due to the anomalous modes. We decompose
the perturbation into four parts lying in: $W_{1}=span\left\{  \cos x,\sin
x\right\}  $ and its complementary subspace $W_{1}^{\perp}\ $in the non-shear
space, $W_{2}=span\left\{  \cos y,\sin y\right\}  $ and its complementary
subspace $W_{2}^{\perp}$ in the shear space. By carefully analyzing the
interaction of these four parts, we can show that the interaction terms in the
nonlinear term $U\cdot\bigtriangledown\omega$ are under control when
$\left\Vert \omega\left(  0\right)  \right\Vert _{L^{2}}=O\left(  \nu\right)
$. As a result, we can still split the non-shear vorticity into the low and
high frequency parts, and treat them separately as for the linearized
equation. Then the nonlinear metastability can be proved. On the square torus,
numerical evidences (\cite{bouchet-simmonet09} \cite{bouchett-reports-09})
suggested that the dipole states of the form $\omega_{0}=\cos x+\cos y$ or
$\sin x+\sin y\ $appear more often in the long time dynamics of 2D Turbulence.
The dipole flows are nonparallel and the enhanced damping problem is much more
subtle to study. At the end of Section \ref{section-metasta}, we discuss some
partial results and difficulties with the dipole states. In particular, in
Proposition \ref{prop-rage-dipole}, we give a RAGE type theorem for the
linearized Euler equation at dipoles.

In our proof of linear and nonlinear enhanced damping for Navier-Stokes
equation, the Hamiltonian structures of the linearized Euler operator play an
important role both in the derivation of the dissipation law
(\ref{dissipation-law-NS}) and in the control of the low frequency part. As a
further application of these Hamiltonian structures, we consider the linear
inviscid damping of more general shear flows $\left(  U\left(  y\right)
,0\right)  $. We study two classes of shear flows. One class is the flows
without inflection points, which are spectrally stable by the classical
Rayleigh criterion. The other class (called class $\mathcal{K}^{+}$) is the
flows $U\left(  y\right)  \ $with one inflection value $U_{s}$ and
$-\frac{U^{\prime\prime}}{U-U_{s}}>0$. These two classes cover all the shear
flows whose nonlinear stability in $L^{2}$ vorticity might be studied by the
energy-Casimir method (see Remark \ref{rmk-energy-casimir}). The flows in the
first class are nonlinearly stable for any $x$ period $2\pi/\alpha$ and are
minimizers of the energy-Casimir functional. The flows in the second class are
stable only when $\alpha>\alpha_{\max}$ for some critical wave number
$\alpha_{\max}$, and are maximizers of the energy-Casimir functional. These
shear flows often appear as long lived coherent states in 2D turbulence. For
example, the Kolmogorov flows which are in class $\mathcal{K}^{+}$.

In Theorem \ref{thm-rage-general-shear}, we give a RAGE\ theorem on the
non-shear subspace $X\ $of $L^{2}\ $for stable shear flows in the first class
and in class $\mathcal{K}^{+}$ with $\alpha>\alpha_{\max}$. As a consequence,
the decay of velocity (in the time average sense) is proved for any non-shear
initial data with $L^{2}$ vorticity. Another consequence is the decay of low
frequency modes in the $L^{2}$ norm of vorticity, which gives a justification
of the dual cascade of 2D turbulence in a weak sense (see Remark
\ref{remark-dual-cascade}). For the critical case $\alpha=\alpha_{\max}$, the
linearized Euler operator $JL$ (defined in (\ref{defn-J-L})) has zero as an
embedded eigenvalue due to the nontrivial $\ker L$. This case is very similar
to the case of bar states on the square torus and can be treated similarly.
The linear damping can be obtained by projecting out $\ker L.$

The flows in class $\mathcal{K}^{+}$ are unstable when $\alpha<\alpha_{\max}$
(\cite{lin-shear}). Moreover, by using an instability index theory recently
developed in \cite{lin-zeng-hamiltonian} for Hamiltonian PDEs, we give an
exact counting formula (Proposition \ref{prop-index-JL}) for the dimension of
unstable modes of the linearized Euler equation. A corollary of this formula
is that $L|_{E^{c}}\geq0$, where $E^{c}$ is the center space corresponding to
the spectra of the linearized Euler operator $JL$ (defined in (\ref{defn-J-L}%
)) on the imaginary axis. Then the RAGE theorem, and as a consequence the
damping of velocity, are obtained for the linearized Euler equation on $E^{c}$.

The inviscid damping was first known for the Couette flow in the 1907 work of
Orr (\cite{orr}). In recent years, the inviscid damping phenomena attracted
new attention. In \cite{Lin-zeng}, it was showed that if we consider initial
(vorticity) perturbation in the Sobolev space $H^{s}$ $\left(  s<\frac{3}%
{2}\right)  $ then the nonlinear damping is not true due to the existence of
nonparallel steady flows of the form of Kelvin's cats eye near Couette. In
\cite{bedrossian-masmoudi}, nonlinear inviscid damping was proved for
perturbations near Couette in Gevrey class (i.e. almost analytic).

The linear inviscid damping for more general shear flows was also recently
studied by some authors. Monotone shears were considered in
\cite{ziillinger-arma} for the case near Couette, and in \cite{wei-zhang-zhao}
for the more general case. The optimal decay rates $O\left(  1/t\right)
,O\left(  1/t^{2}\right)  $ for the horizontal and vertical velocities were
obtained for initial vorticity in $H^{1}$ and $H^{2}$ respectively. In
\cite{wei-zhang-zhao-2}, general shear flows satisfying some nondegeneracy
conditions were considered, and certain space-time estimates for velocities of
the linearized Euler equation were obtained for initial vorticity in $H^{1}$.
The optimal decay rates were also obtained in \cite{wei-zhang-zhao-2} for a
special class of symmetric shear flows. The non-existence of embedded
eigenvalues was assumed in above works.

We comment on some differences of our results on inviscid damping with the
previous work. First, for the two classes of shear flows we considered, we do
not need to assume the non-existence of embedded eigenvalues. This assumption
is proved to be true for flows without inflection points and for flows in
class $\mathcal{K}^{+}$ with $\alpha>\alpha_{\max}$. But for $\alpha
=\alpha_{\max}$ and some $\alpha<\alpha_{\max}$, zero is indeed an embedded
eigenvalue. For these cases, the inviscid damping can still be proved as in
Corollary \ref{cor-decay-general-shear} ii) and Theorem
\ref{thm-damping-center} ii), as well as for the Kolmogorov flow on the square
torus. Second, the inviscid damping results we obtained are for the initial
vorticity in $L^{2}$. In \cite{wei-zhang-zhao} \cite{wei-zhang-zhao-2},
initial vorticity with higher regularity was considered and the linear damping
for $L^{2}$ vorticity was not studied. In \cite{wei-zhang-zhao}
\cite{wei-zhang-zhao-2} \cite{ziillinger-arma}, the linearized Euler equation
was studied in a channel. Here, we treat the cases of the channel and tori in
a unified way. In some sense, the RAGE theorem type results imply more
information than just the damping of velocities. For example, the decay of low
frequency part of the vorticity does not follow from the decay of velocity.
Third, our approach, which exploits the Hamiltonian structures of the Euler
equation, does not rely on ODE techniques. Therefore, it could be used for the
problems involving nonparallel flows, see Proposition \ref{prop-rage-dipole}
for dipoles and Theorem 11.7 in \cite{lin-zeng-hamiltonian} for general steady
Euler flows. Moreover, more information on the damping could be derived from
the regularity properties of the spectral measure of $JL$ (see Remark
\ref{rmk-regularity-measure}). This might provide an alternative approach to
study the inviscid damping in other problems.

This paper is organized as follows. In Section 2, we study the linear enhanced
damping for the linearized Navier-Stokes equation. In Section 3, the nonlinear
enhanced damping (i.e. metastability of Kolmogorov flows) is proved for the
nonlinear Navier-Stokes equation. We discuss the cases of rectangular and
square tori separately in Sections 2 and 3. In Section 4, the linear inviscid
damping is proved for both stable and unstable shear flows.

\section{Linear enhanced damping}

In this section, we prove Theorem \ref{thm-linearized} on the enhanced damping
for the linearized Navier-Stokes equation (\ref{eqn-bar-LNS}). We consider the
cases of rectangular and square tori separately.

\subsection{Linearized Navier-Stokes on a rectangular torus}

Consider the linearized equation (\ref{eqn-bar-LNS}) on a torus
\[
\mathbb{T}_{\alpha}=\left\{  0<y<2\pi,0<x<\frac{2\pi}{\alpha}\right\}
,\ \alpha>1.
\]
We divide the proof of Theorem \ref{thm-linearized} i) into several steps. In
the proof, we shall use $C$ to denote a generic constant in the estimates.
First, we prove the dissipation law (\ref{dissipation-law-NS}).

\begin{lemma}
\label{lemma-dissi-linear}Let $\omega\left(  t\right)  $ be a solution of
(\ref{eqn-bar-LNS}) with the initial data $\omega\left(  t\right)  \in
L^{2}\left(  \mathbb{T}_{\alpha}\right)  $. Then
\begin{equation}
\frac{d}{dt}\int_{\mathbb{T}_{\alpha}}(|\omega|^{2}-|\bigtriangledown\psi
|^{2})dxdy=-2\nu\int_{\mathbb{T}_{\alpha}}(|\bigtriangledown\omega
|^{2}-|\omega|^{2})dxdy, \label{dissipation-law-NS-lemma}%
\end{equation}
for any $t>0$.
\end{lemma}

\begin{proof}
The equation (\ref{eqn-bar-LNS}) can be written as
\[
\partial_{t}\omega=\nu\Delta\omega+e^{-\nu t}JL\omega,
\]
where $J,L$ are defined in (\ref{defn-J-L}). Thus we have
\begin{align*}
\frac{d}{dt}\int_{\mathbb{T}_{\alpha}}(|\omega|^{2}-|\bigtriangledown\psi
|^{2})dxdy  &  =\frac{d}{dt}\left\langle L\omega,\omega\right\rangle
=2\left\langle L\omega,\omega_{t}\right\rangle \\
&  =e^{-\nu t}\left\langle L\omega,JL\omega\right\rangle +2\int_{\mathbb{T}%
_{\alpha}}\nu\Delta\omega\left(  \omega-\psi\right)  dxdy\\
&  =-2\nu\int_{\mathbb{T}_{\alpha}}(|\bigtriangledown\omega|^{2}-|\omega
|^{2})dxdy.
\end{align*}
In the last equality above, we use integration by parts and the fact that $J$
is anti-selfadjoint.
\end{proof}

In the next lemma, we show that the quadratic forms on both sides of
(\ref{dissipation-law-NS-lemma}) are positive definite for a non-shear
vorticity $\omega$.

\begin{lemma}
Let $\alpha>1$ and $\omega\in X\cap H^{1}\left(  \mathbb{T}_{\alpha}\right)
$. Then there exists a constant $c_{0}>0$ depending only on $\alpha$ such
that
\begin{equation}
\int_{\mathbb{T}_{\alpha}}(|\omega|^{2}-|\bigtriangledown\psi|^{2})dxdy\geq
c_{0}\left\Vert \omega\right\Vert _{L^{2}}^{2}, \label{positive-L2}%
\end{equation}
and
\begin{equation}
\int_{\mathbb{T}_{\alpha}}(|\bigtriangledown\omega|^{2}-|\omega|^{2})dxdy\geq
c_{0}\left\Vert \omega\right\Vert _{H^{1}}^{2}. \label{positive-H1}%
\end{equation}

\end{lemma}

\begin{proof}
For any
\[
\omega=\sum_{k\neq0}\omega_{k}\left(  y\right)  e^{ik\alpha x}\in H^{1}\left(
\mathbb{T}_{\alpha}\right)  ,
\]
we have
\begin{align*}
\int_{\mathbb{T}_{\alpha}}(|\omega|^{2}-|\bigtriangledown\psi|^{2})dxdy  &
=\sum_{0\neq k\in\mathbf{Z}}\left\langle \left(  1-\left(  -\frac{d^{2}%
}{dy^{2}}+\alpha^{2}k^{2}\right)  ^{-1}\right)  \omega_{k},\omega
_{k}\right\rangle \\
&  \geq\left(  1-\alpha^{-2}\right)  \sum_{0\neq k\in\mathbf{Z}}\left(
\omega_{k},\omega_{k}\right)  =\left(  1-\alpha^{-2}\right)  \left\Vert
\omega\right\Vert _{L^{2}}^{2},
\end{align*}
and
\begin{align*}
\int_{\mathbb{T}_{\alpha}}(|\bigtriangledown\omega|^{2}-|\omega|^{2})dxdy  &
=\sum_{0\neq k\in\mathbf{Z}}\left\langle \left(  -\frac{d^{2}}{dy^{2}}%
+\alpha^{2}k^{2}-1\right)  \omega_{k},\omega_{k}\right\rangle \\
&  \geq\sum_{0\neq k\in\mathbf{Z}}\int\left\vert \omega_{k}^{\prime}\left(
y\right)  \right\vert ^{2}dy+\left(  \alpha^{2}-1\right)  \int\left\vert
\omega_{k}\left(  y\right)  \right\vert ^{2}dy\\
&  \geq\min\left\{  1,\alpha^{2}-1\right\}  \left\Vert \omega\right\Vert
_{H^{1}}^{2}.
\end{align*}

\end{proof}

Next, we study the linearized Euler equation at the Kolmogorov flow
\begin{equation}
\partial_{t}\omega=-\sin y\partial_{x}\left(  1+\Delta^{-1}\right)
\omega=JL\omega. \label{eqn-linearized Euler}%
\end{equation}

\begin{lemma}
\label{lemma-Euler-growth}Let $\omega\left(  t\right)  $ be a solution of
(\ref{eqn-linearized Euler}) with $\omega\left(  0\right)  \in X\cap
H^{1}\left(  \mathbb{T}_{\alpha}\right)  $. Then
\[
\left\Vert \omega\left(  t\right)  \right\Vert _{H^{1}}\leq C\left(
1+t\right)  \left\Vert \omega\left(  0\right)  \right\Vert _{H^{1}},
\]
for some constant $C$.
\end{lemma}

\begin{proof}
First, we note that
\[
\left\langle L\omega,\omega\right\rangle =\int_{\mathbb{T}_{\alpha}}%
(|\omega|^{2}-|\bigtriangledown\psi|^{2})dxdy
\]
is conserved for (\ref{eqn-bar-LNS}). Therefore by the positivity estimate
(\ref{positive-L2}), we have
\begin{equation}
\left\Vert \omega\left(  t\right)  \right\Vert _{L^{2}}=\left\Vert
e^{tJL}\omega\left(  0\right)  \right\Vert _{L^{2}}\leq C\left\Vert
\omega\left(  0\right)  \right\Vert _{L^{2}},\ \label{energy-bound-L2}%
\end{equation}
for some constant $C$. Taking $\partial_{x}$ of (\ref{eqn-linearized Euler}),
we have
\[
\partial_{t}\partial_{x}\omega=-\sin y\partial_{x}\left(  1+\Delta
^{-1}\right)  \partial_{x}\omega
\]
and therefore
\[
\left\Vert \partial_{x}\omega\left(  t\right)  \right\Vert _{L^{2}}=\left\Vert
e^{tJL}\partial_{x}\omega\left(  0\right)  \right\Vert _{L^{2}}\leq
C\left\Vert \partial_{x}\omega\left(  0\right)  \right\Vert _{L^{2}}.
\]
Taking $\partial_{y}$ of (\ref{eqn-linearized Euler}), we have
\[
\partial_{t}\partial_{y}\omega=-\sin y\partial_{x}\left(  1+\Delta
^{-1}\right)  \partial_{y}\omega-\cos y\partial_{x}\left(  1+\Delta
^{-1}\right)  \omega,
\]
and
\[
\partial_{y}\omega\left(  t\right)  =e^{tJL}\partial_{y}\omega\left(
0\right)  -\int_{0}^{t}e^{\left(  t-s\right)  JL}\cos y\partial_{x}\left(
1+\Delta^{-1}\right)  \omega\left(  s\right)  ds.
\]
Therefore
\begin{align*}
\left\Vert \partial_{y}\omega\left(  t\right)  \right\Vert _{L^{2}}  &  \leq
C\left(  \left\Vert \partial_{y}\omega\left(  0\right)  \right\Vert _{L^{2}%
}+\int_{0}^{t}\left\Vert \partial_{x}\left(  1+\Delta^{-1}\right)
\omega\left(  s\right)  \right\Vert _{L^{2}}ds\right) \\
&  \leq C\left(  1+t\right)  \left\Vert \nabla\omega\left(  0\right)
\right\Vert _{L^{2}}.
\end{align*}
This finishes the proof of the lemma.
\end{proof}

In the next lemma, we study the spectral properties of the linearized Euler
operator $A=JL$.

\begin{lemma}
\label{lemma-spectra-rectangular}i) The operator $A:X^{\ast}\rightarrow X$ is
anti-selfadjoint in the inner product $\left[  \cdot,\cdot\right]
=\left\langle L\cdot,\cdot\right\rangle .$

ii) The spectrum of $A$ lies on the imaginary axis and is purely continuous.
\end{lemma}

\begin{proof}
i) By the positivity of $L$ on $X$, $\left[  \cdot,\cdot\right]  =\left\langle
L\cdot,\cdot\right\rangle $ defines an equivalent inner product to the $L^{2}$
inner product. For any $\omega_{1},\omega_{2}\in X$, we have
\begin{align*}
\left[  A\omega_{1},\omega_{2}\right]   &  =\left\langle LJL\omega_{1}%
,\omega_{2}\right\rangle =\left\langle JL\omega_{1},L\omega_{2}\right\rangle
=-\left\langle L\omega_{1},JL\omega_{2}\right\rangle \\
&  =-\left[  \omega_{1},A\omega_{2}\right]  ,
\end{align*}
and thus $A$ is anti-selfadjoint on $\left(  X,\left[  \cdot,\cdot\right]
\right)  $.

ii) By property i), the spectrum of $A$ in $L^{2}$ is on the imaginary axis.
Since $A=-\sin y\partial_{x}\left(  1+\Delta^{-1}\right)  $ is a compact
perturbation of $D=-\sin y\partial_{x}$, by Weyl's Theorem the continuous
spectrum of $A$ is the same as that of $D$, which is clearly the whole
imaginary axis. It remains to show that there is no embedded eigenvalue of $A$
on the imaginary axis. Suppose $A\omega=\lambda\omega$, where $\lambda\in
i\mathbf{R}$ and $0\neq\omega\in X$. Let
\[
\omega=\sum_{0\neq k\in\mathbf{Z}}\omega_{k}\left(  y\right)  e^{ik\alpha
x},\ \ \ \psi_{k}=\left(  -\frac{d^{2}}{dy^{2}}+\alpha^{2}k^{2}\right)
^{-1}\omega_{k}.
\]
Then if $\omega_{k}\neq0$, we have
\[
ik\alpha\sin y\left(  \omega_{k}-\psi_{k}\right)  =\lambda\omega_{k},
\]
which is equivalent to the Rayleigh equation
\[
\left(  -\frac{d^{2}}{dy^{2}}+\alpha^{2}k^{2}-\frac{\sin y}{\sin y-c}\right)
\psi_{k}=0,
\]
with $c=\frac{\lambda}{ik\alpha}\in\mathbf{R}$. Since $\omega_{k}\in L^{2}$
implies $\psi_{k}\in H^{2}$, by Lemma \ref{lemma-inflection-value} in the
Appendix, we must have $c=0$ which is the only inflection value of $\sin y$.
Thus
\[
\left(  -\frac{d^{2}}{dy^{2}}+\alpha^{2}k^{2}-1\right)  \psi_{k}=0,
\]
which implies that $\psi_{k}=0$ since $\alpha>1$. This contradiction rules out
the embedded imaginary eigenvalues of $A$ in $X$.
\end{proof}

By the above Lemma and the RAGE theorem, we have

\begin{corollary}
Let $B$ be any compact operator in $L^{2}\left(  \mathbb{T}_{\alpha}\right)
$. Then
\[
\frac{1}{T}\int_{0}^{T}\left\Vert B\omega\left(  t\right)  \right\Vert
_{L^{2}}^{2}dt\rightarrow0,\ \text{when }T\rightarrow\infty,
\]
for any solution $\omega\left(  t\right)  \ $of (\ref{eqn-linearized Euler})
with $\omega\left(  0\right)  \in X$.
\end{corollary}

For the proof of the enhanced damping, we need a more quantitative version of
RAGE\ theorem. Let $\alpha^{2}\leq\lambda_{1}\leq\lambda_{2}\leq\ldots$ be the
eigenvalues of the operator $-\bigtriangleup$ on $X\ $and $e_{1},e_{2}\ldots$
be the corresponding orthonormal eigenvectors. Let $P_{N}$ denote the
orthogonal projection on the subspace spanned by the first $N$ eigenvectors
$e_{1},e_{2},\ldots,e_{N}$ and $S=\{\omega\in X:\Vert\omega\Vert_{L^{2}}=1\}$
be the unit sphere in $X$. Denote the norms
\[
\left\Vert \omega\right\Vert _{X}^{2}=\left\langle \left(  1+\Delta
^{-1}\right)  \omega,\omega\right\rangle =\int_{\mathbb{T}_{\alpha}}%
(|\omega|^{2}-|\bigtriangledown\psi|^{2})dxdy,
\]%
\[
\left\Vert \omega\right\Vert _{X^{1}}^{2}=\left\langle \left(  -\Delta
-1\right)  \omega,\omega\right\rangle =\int_{\mathbb{T}_{\alpha}%
}(|\bigtriangledown\omega|^{2}-|\omega|^{2})dxdy
\]
which are equivalent to $L^{2}$ and $H^{1}\ $norms on $X.$ Let $\omega
=\sum_{k\geq1}c_{k}e_{k}$, then
\[
\left\Vert \omega\right\Vert _{X}^{2}=\sum_{k\geq1}\left(  1-\frac{1}%
{\lambda_{k}}\right)  \left\vert c_{k}\right\vert ^{2},\ \
\]
and
\[
\left\Vert \omega\right\Vert _{X^{1}}^{2}=\sum_{k\geq1}\left(  \lambda
_{k}-1\right)  \left\vert c_{k}\right\vert ^{2}.
\]

The following version of the RAGE theorem is obtained as in
\cite{kiselev-annals-mixing}.

\begin{lemma}
\label{lemma-rage-quant}Let $K\subset S$ be a compact set and $J,L$ are
defined in (\ref{defn-J-L}). For any $N,\kappa>0$, there exists $T_{c}%
(N,\kappa,K)$ such that for all $T\geq T_{c}$ and any $\omega\left(  0\right)
\in K$,
\[
\frac{1}{T}\int_{0}^{T}\Vert P_{N}e^{tJL}\omega\left(  0\right)  \Vert_{X}%
^{2}dt\leq\kappa\Vert\omega\left(  0\right)  \Vert_{X}^{2}.
\]

\end{lemma}

Now we estimate the difference of the solutions of the linearized NS and Euler equations.

\begin{lemma}
\label{lemma-difference-NS-Euler}Let $\omega^{\nu},\omega^{0}\ $be the
solutions of the linearized NS equation (\ref{eqn-bar-LNS}) and Euler equation
(\ref{eqn-linearized Euler}) with the initial data in $X\cap H^{1}$. Then
there exists some constant $C_{0}>0$ such that
\begin{equation}
\frac{d}{dt}\left\Vert \omega^{\nu}-\omega^{0}\right\Vert _{X}^{2}\leq
C_{0}\nu\left(  1+t^{2}\right)  \left\Vert \omega^{0}\left(  t\right)
\right\Vert _{H^{1}}^{2}, \label{error-ineq-1}%
\end{equation}
for all $t\in\left(  0,+\infty\right)  $.
\end{lemma}

\begin{proof}
Let $\psi^{\nu},\psi^{0}$ be the corresponding stream functions. Denote
$\omega=\omega^{\nu}-\omega^{0}$ and $\psi=\psi^{\nu}-\psi^{0}$, then
\[
\omega_{t}+e^{-\nu t}\sin y\partial_{x}(\omega-\psi)+(e^{-\nu t}-1)\sin
y\partial_{x}(\omega^{0}-\psi^{0})-\nu\bigtriangleup\omega^{\nu}=0.
\]
We have
\[
\ \ \ \frac{d}{dt}\frac{1}{2}\left\Vert \omega^{\nu}-\omega^{0}\right\Vert
_{X}^{2}=\frac{d}{dt}\frac{1}{2}\int_{\mathbb{T}_{\alpha}}(|\omega
|^{2}-|\bigtriangledown\psi|^{2})dxdy=\int_{\mathbb{T}_{\alpha}}\omega
_{t}\left(  \omega-\psi\right)  dxdy
\]%
\begin{align*}
&  =-\int_{\mathbb{T}_{\alpha}}(e^{-\nu t}-1)\sin y\partial_{x}(\omega
^{0}-\psi^{0})\left(  \omega-\psi\right)  dxdy+\nu\int_{\mathbb{T}_{\alpha}%
}\bigtriangleup\omega^{\nu}(\omega^{\nu}-\psi^{\nu})dxdy\\
&  \ \ \ \ \ \ \ \ \ -\nu\int_{\mathbb{T}_{\alpha}}\bigtriangleup\omega^{\nu
}(\omega^{0}-\psi^{0})dxdy\\
&  =I+II+III.
\end{align*}
Since $0\leq1-e^{-\nu t}\leq\nu t$ when $t>0$, we have
\begin{align*}
I  &  =\left(  1-e^{-\nu t}\right)  \int_{\mathbb{T}_{\alpha}}\sin
y\partial_{x}(\omega^{0}-\psi^{0})\left(  \omega^{\nu}-\psi^{\nu}\right)
dxdy\\
&  =-\left(  1-e^{-\nu t}\right)  \int_{\mathbb{T}_{\alpha}}\sin y\partial
_{x}(\omega^{\nu}-\psi^{\nu})\left(  \omega^{0}-\psi^{0}\right)  dxdy\\
&  \leq C\nu t\left\Vert \omega^{\nu}\right\Vert _{H^{1}}\left\Vert \omega
^{0}\right\Vert _{L^{2}}.
\end{align*}
By integration by parts and (\ref{positive-H1}),
\[
II=-\nu\int_{\mathbb{T}_{\alpha}}(|\bigtriangledown\omega^{\nu}|^{2}%
-|\omega^{\nu}|^{2})dxdy\leq-c_{0}\nu\left\Vert \omega^{\nu}\right\Vert
_{H^{1}}^{2}.
\]
For the last term, we have
\[
III=\nu\int_{\mathbb{T}_{\alpha}}\nabla\omega^{\nu}\cdot\nabla(\omega^{0}%
-\psi^{0})dxdy\leq\nu\left\Vert \omega^{\nu}\right\Vert _{H^{1}}\left\Vert
\omega^{0}\right\Vert _{H^{1}}.
\]
Combining above, we get
\begin{align*}
\ \ \ \  &  \ \ \ \ \ \frac{d}{dt}\frac{1}{2}\int_{\mathbb{T}_{\alpha}%
}(|\omega|^{2}-|\bigtriangledown\psi|^{2})dxdy\\
&  \leq\nu\left(  -c_{0}\left\Vert \omega^{\nu}\right\Vert _{H^{1}}%
^{2}+\left\Vert \omega^{\nu}\right\Vert _{H^{1}}\left(  Ct\left\Vert
\omega^{0}\right\Vert _{L^{2}}+\left\Vert \omega^{0}\right\Vert _{H^{1}%
}\right)  \right) \\
&  \leq C\nu\left(  1+t^{2}\right)  \left\Vert \omega^{0}\left(  t\right)
\right\Vert _{H^{1}}^{2}.
\end{align*}
This proves (\ref{error-ineq-1}).
\end{proof}

As a corollary, combining (\ref{error-ineq-1}), Lemma \ref{lemma-Euler-growth}
and (\ref{positive-H1}), we have
\begin{equation}
\left\Vert \omega^{\nu}\left(  t\right)  -\omega^{0}\left(  t\right)
\right\Vert _{X}^{2}\leq\left\Vert \omega^{\nu}\left(  0\right)  -\omega
^{0}\left(  0\right)  \right\Vert _{X}^{2}+C_{1}\nu\left(  1+t^{5}\right)
\left\Vert \omega^{0}\left(  0\right)  \right\Vert _{X^{1}}^{2}%
,\ \label{error-inequ-2}%
\end{equation}
for some constant $C_{1}>0.\ $

We are now ready to prove the linear enhanced damping.

\begin{proof}
[Proof of Theorem \ref{thm-linearized} i)]The proof follows by the same
arguments in \cite{kiselev-annals-mixing}. We sketch it here. Fixed
$\delta,\tau>0$, we choose $N$ such that
\[
\exp\left(  -2\lambda_{N}\tau\right)  <c_{0}\delta^{2},
\]
where $c_{0}$ is the constant in (\ref{positive-L2}). Define a compact set
\[
K=span\left\{  e_{1},\cdots,e_{N}\right\}  =R\left(  P_{N}\right)  .
\]
Denote $t_{1}=T_{c}(N,\frac{1}{10},K)$ as in Lemma \ref{lemma-rage-quant} and
let $\nu\left(  \delta,\tau\right)  $ be such that
\[
\nu\left(  \delta,\tau\right)  C_{1}\left(  1+t_{1}^{5}\right)  <\frac
{1}{10\lambda_{N}},
\]
where $C_{1}$ is the constant in the estimate (\ref{error-inequ-2}). For
$0<\nu<\nu\left(  \delta,\tau\right)  $, suppose that
\[
\left\Vert \omega^{\nu}\left(  t\right)  \right\Vert _{X^{1}}^{2}>\lambda
_{N}\left\Vert \omega^{\nu}\left(  t\right)  \right\Vert _{X}^{2}%
\]
is true for $t$ in some interval $\left(  a,b\right)  \subset\left(
0,\tau/\nu\right)  $. Then by (\ref{dissipation-law-NS-lemma}), we have
\begin{equation}
\left\Vert \omega^{\nu}\left(  b\right)  \right\Vert _{X}^{2}\leq\exp\left(
-\nu\lambda_{N}\left(  b-a\right)  \right)  \left\Vert \omega^{\nu}\left(
a\right)  \right\Vert _{X}^{2}. \label{decay-case1}%
\end{equation}
Now consider any $t_{0}\in\left(  0,\tau/\nu\right)  $ satisfying
\[
\left\Vert \omega^{\nu}\left(  t_{0}\right)  \right\Vert _{X^{1}}^{2}%
\leq\lambda_{N}\left\Vert \omega^{\nu}\left(  t_{0}\right)  \right\Vert
_{X}^{2}.
\]
Denote $\omega_{0}=\omega^{\nu}\left(  t_{0}\right)  $ and let $\omega
^{0}\left(  t\right)  $ $\left(  t\in\left[  t_{0},t_{0}+t_{1}\right]
\right)  \ $be the solution of (\ref{eqn-linearized Euler}) with $\omega
^{0}\left(  t_{0}\right)  =\omega_{0}$. By the choice of $t_{0},\ \nu\left(
\delta,\tau\right)  $ and (\ref{error-inequ-2}), we have
\begin{equation}
\left\Vert \omega^{\nu}\left(  t\right)  -\omega^{0}\left(  t\right)
\right\Vert _{X}^{2}\leq\frac{1}{10}\left\Vert \omega_{0}\right\Vert _{X}%
^{2},\ \forall\ t\in\left[  t_{0},t_{0}+t_{1}\right]  . \label{error-ineq-3}%
\end{equation}
By the definition of $t_{1}$, we have
\[
\frac{1}{t_{1}}\int_{t_{0}}^{t_{0}+t_{1}}\left\Vert P_{N}\omega^{0}\left(
t\right)  \right\Vert _{X}^{2}dt\leq\frac{1}{10}\left\Vert \omega
_{0}\right\Vert _{X}^{2}.
\]
Since $\left\Vert \omega^{0}\left(  t\right)  \right\Vert _{X}=\left\Vert
\omega_{0}\right\Vert _{X}$ by the conservation of $\left\langle
L\omega,\omega\right\rangle $ for the equation (\ref{eqn-linearized Euler}),
it follows that
\[
\frac{1}{t_{1}}\int_{t_{0}}^{t_{0}+t_{1}}\left\Vert \left(  1-P_{N}\right)
\omega^{0}\left(  t\right)  \right\Vert _{X}^{2}dt\geq\frac{9}{10}\left\Vert
\omega_{0}\right\Vert _{X}^{2}.
\]
Combined with (\ref{error-ineq-3}), above implies that
\[
\frac{1}{t_{1}}\int_{t_{0}}^{t_{0}+t_{1}}\left\Vert \left(  1-P_{N}\right)
\omega^{\nu}\left(  t\right)  \right\Vert _{X}^{2}dt\geq\frac{1}{2}\left\Vert
\omega_{0}\right\Vert _{X}^{2}.
\]
For $\omega^{\nu}\left(  t\right)  =\sum_{k\geq1}c_{k}e_{k}$, we have
\begin{align*}
\left\Vert \omega^{\nu}\left(  t\right)  \right\Vert _{X^{1}}^{2}  &
=\sum_{k\geq1}\left(  \lambda_{k}-1\right)  \left\vert c_{k}\right\vert
^{2}\geq\frac{1}{\lambda_{N}}\sum_{k\geq N+1}\left(  1-\frac{1}{\lambda_{k}%
}\right)  \left\vert c_{k}\right\vert ^{2}\\
&  =\frac{1}{\lambda_{N}}\left\Vert \left(  1-P_{N}\right)  \omega^{\nu
}\left(  t\right)  \right\Vert _{X}^{2},
\end{align*}
and thus
\[
\int_{t_{0}}^{t_{0}+t_{1}}\left\Vert \omega^{\nu}\left(  t\right)  \right\Vert
_{X^{1}}^{2}dt\geq\frac{\lambda_{N}t_{1}}{2}\left\Vert \omega_{0}\right\Vert
_{X}^{2}=\frac{\lambda_{N}t_{1}}{2}\left\Vert \omega^{\nu}\left(
t_{0}\right)  \right\Vert _{X}^{2}.
\]
Then (\ref{dissipation-law-NS-lemma}) implies that
\begin{align}
\left\Vert \omega^{\nu}\left(  t_{0}+t_{1}\right)  \right\Vert _{X}^{2}  &
\leq\left\Vert \omega^{\nu}\left(  t_{0}\right)  \right\Vert _{X}^{2}-2\nu
\int_{t_{0}}^{t_{0}+t_{1}}\left\Vert \omega^{\nu}\left(  t\right)  \right\Vert
_{X^{1}}^{2}dt\label{decay-case2}\\
&  \leq\left(  1-\lambda_{N}\nu t_{1}\right)  \left\Vert \omega^{\nu}\left(
t_{0}\right)  \right\Vert _{X}^{2}\leq e^{-\lambda_{N}\nu t_{1}}\left\Vert
\omega^{\nu}\left(  t_{0}\right)  \right\Vert _{X}^{2}.\nonumber
\end{align}
We can split the interval $\left[  0,\frac{\tau}{\nu}\right]  $ into a union
of intervals such that either (\ref{decay-case1}) or (\ref{decay-case2}) holds
true. Therefore we have
\[
\left\Vert \omega^{\nu}\left(  \frac{\tau}{\nu}\right)  \right\Vert _{X}%
^{2}\leq e^{-\lambda_{N}\tau}\left\Vert \omega^{\nu}\left(  0\right)
\right\Vert _{X}^{2}<c_{0}\delta^{2}\left\Vert \omega^{\nu}\left(  0\right)
\right\Vert _{X}^{2}%
\]
and by (\ref{positive-L2})
\[
\left\Vert \omega^{\nu}\left(  \frac{\tau}{\nu}\right)  \right\Vert _{L^{2}%
}^{2}\leq\frac{1}{c_{0}}\left\Vert \omega^{\nu}\left(  \frac{\tau}{\nu
}\right)  \right\Vert _{X}^{2}<\delta^{2}\left\Vert \omega^{\nu}\left(
0\right)  \right\Vert _{X}^{2}<\delta^{2}\left\Vert \omega^{\nu}\left(
0\right)  \right\Vert _{L^{2}}^{2}\text{.}%
\]
This finishes the proof of Theorem \ref{thm-linearized} i).
\end{proof}

\subsection{Linearized Navier-Stokes on a square torus}

Now we consider the linearized equation (\ref{eqn-bar-LNS}) on the square
torus
\[
\mathbb{T}=\left\{  0<y<2\pi,0<x<2\pi\right\}  .
\]
In this case, there is a two dimensional kernel space $W_{1}$ spanned by
$\left\{  \cos x,\sin x\right\}  $ of the operator $L=1+\Delta^{-1}$ on the
non-shear space $X$. We will sketch the changes induced by these anomalous
modes, in the proof of Theorem \ref{thm-linearized} ii).

First, we note that $L$ is positive on the space $X_{1}=\left(  I-P_{1}%
\right)  X$, where $P_{1}$ is the projection of $X$ to $W_{1}$. Let
$\omega\left(  t\right)  $ be the solution of (\ref{eqn-bar-LNS}) with any
initial data $\omega\left(  0\right)  \in X$. Then $\omega_{1}=\left(
I-P_{1}\right)  \omega$ satisfies the equation
\begin{equation}
\partial_{t}\omega_{1}=\nu\Delta\omega_{1}-e^{-\nu t}\left[  \left(
I-P_{1}\right)  \sin y\partial_{x}\left(  1+\Delta^{-1}\right)  \right]
\omega_{1}. \label{linearized-NS-projected}%
\end{equation}
It is easy to check that the same dissipation law
\[
\frac{d}{dt}\int_{\mathbb{T}}(|\omega_{1}|^{2}-|\bigtriangledown\psi_{1}%
|^{2})dxdy=-2\nu\int_{\mathbb{T}}(|\bigtriangledown\omega_{1}|^{2}-|\omega
_{1}|^{2})dxdy,
\]
holds true for (\ref{linearized-NS-projected}). Moreover, there exists
$c_{0}>0$ such that
\begin{equation}
\int_{\mathbb{T}}(|\omega|^{2}-|\bigtriangledown\psi|^{2})dxdy\geq
c_{0}\left\Vert \omega\right\Vert _{L^{2}}^{2}, \label{positivity-L2-X1}%
\end{equation}
and
\[
\int_{\mathbb{T}}(|\bigtriangledown\omega|^{2}-|\omega|^{2})dxdy\geq
c_{0}\left\Vert \omega\right\Vert _{H^{1}}^{2},
\]
for $\omega\in X_{1}$. Define the operator $A_{1}:\left(  X_{1}\right)
^{\ast}\rightarrow X_{1}$ by
\[
A_{1}=-\left(  I-P_{1}\right)  \sin y\partial_{x}\left(  1+\Delta^{-1}\right)
=\left(  I-P_{1}\right)  JL.
\]
Since $L|_{X_{1}}>0$, $\left[  \cdot,\cdot\right]  =\left\langle L\cdot
,\cdot\right\rangle $ is again an equivalent inner product on $X_{1}$ and
$A_{1}$ is anti-selfadjoint on $\left(  X_{1},\left[  \cdot,\cdot\right]
\right)  $. Indeed, for any $w_{1},w_{2}\in X_{1}$%
\begin{align*}
\left[  A_{1}w_{1},w_{2}\right]   &  =\left\langle L\left(  I-P_{1}\right)
JLw_{1},w_{2}\right\rangle =\left\langle LJLw_{1},w_{2}\right\rangle
=-\left\langle Lw_{1},JLw_{2}\right\rangle \\
&  =-\left\langle Lw_{1},\left(  I-P_{1}\right)  JLw_{2}\right\rangle
=-\left[  w_{1},A_{1}w_{2}\right]  .
\end{align*}
Therefore the spectrum of $A_{1}$ lies on the imaginary axis. We will show
that $A_{1}$ has no embedded imaginary eigenvalues.

\begin{lemma}
\label{lemma-continuous-A-torus}The spectrum of $A_{1}$ is purely continuous.
\end{lemma}

\begin{proof}
Suppose $A_{1}$ has an eigenvalue $\lambda\in i\mathbf{R}$ and $A_{1}%
\omega=\lambda\omega,$ where $0\neq\omega\in X_{1}$. Then
\[
A\omega-\lambda\omega=JL\omega-\lambda\omega=\tilde{\omega}\in W_{1}.
\]
If $\lambda\neq0$, by noting that $A\tilde{\omega}=0$, we get
\[
A\left(  \omega+\frac{1}{\lambda}\tilde{\omega}\right)  =\lambda\left(
\omega+\frac{1}{\lambda}\tilde{\omega}\right)  .
\]
That is, $\lambda$ is an eigenvalue of $A$. This is a contradiction, since by
the proof of Lemma \ref{lemma-spectra-rectangular}, the operator $A$ has no
nonzero eigenvalues. If $\lambda=0$, then we must have $A\omega=\tilde{\omega
}$ for some $0\neq\tilde{\omega}\in W_{1}$, since $\omega\in X_{1}$ implies
that $A\omega\neq0$. Let $\tilde{\omega}=c_{1}e^{ix}+c_{-1}e^{-ix}\in W_{1}$
and
\[
\ \ \ \psi=\Delta^{-1}\omega=a_{1}\left(  y\right)  e^{ix}+a_{-1}\left(
y\right)  e^{-ix}.
\]
$\ $From the equation $A\omega=\sin y\partial_{x}(\omega+\psi)=\tilde{\omega}%
$, we get
\[
a_{1}^{\prime\prime}\left(  y\right)  =\frac{c_{1}}{i\sin y},\ \ a_{-1}%
^{\prime\prime}\left(  y\right)  =\frac{c_{-1}}{-i\sin y},
\]
and thus $\psi\notin H^{2}$. This shows that $0$ is not an eigenvalue of
$A_{1}$ and the proof of the lemma is finished.
\end{proof}

By the above lemma, we can use the RAGE theorem for the semigroup $e^{tA_{1}}$
on $X_{1}$, which corresponds to solutions of the projected linearized Euler
equation%
\begin{equation}
\partial_{t}\omega=\left(  I-P_{1}\right)  \sin y\partial_{x}\left(
1+\Delta^{-1}\right)  \omega. \label{linearized Euler-projected}%
\end{equation}
In particular, let $P_{N}$ be the projection of $L^{2}$ to the space spanned
by the first $N$ eigenfunction of $-\bigtriangleup\ $on $X_{1}$, then we have:
For any $N,\kappa>0$, there exists $T_{c}(N,\kappa,K)$ such that for all
$T\geq T_{c}$ and any $\omega\left(  0\right)  \in R\left(  P_{N}\right)  $,
\begin{equation}
\frac{1}{T}\int_{0}^{T}\Vert P_{N}e^{tJL}\omega\left(  0\right)  \Vert_{X}%
^{2}dt\leq\kappa\Vert\omega\left(  0\right)  \Vert_{X}^{2}. \label{rage-torus}%
\end{equation}

In the next Lemma, we obtain the same estimates on the growth of solutions of
(\ref{linearized Euler-projected}), as in Lemma \ref{lemma-Euler-growth}.

\begin{lemma}
\label{lemma-growth-euler-torus}Let $\omega\left(  t\right)  $ be a solution
of (\ref{linearized Euler-projected}) with $\omega\left(  0\right)  \in
X_{1}\cap H^{1}\left(  \mathbb{T}\right)  $. Then
\[
\left\Vert \omega\left(  t\right)  \right\Vert _{H^{1}}\leq C\left(
1+t\right)  \left\Vert \omega\left(  0\right)  \right\Vert _{H^{1}},
\]
for some constant $C$.
\end{lemma}

\begin{proof}
The proof is very similar to that of Lemma \ref{lemma-Euler-growth}. We only
sketch it briefly. By the conservation of $\left\langle L\omega,\omega
\right\rangle $ for the equation (\ref{linearized Euler-projected}) and the
positivity of $L|_{X_{1}}$, the $L^{2}$ norm of $\omega\left(  t\right)  $ is
bounded by $\omega\left(  0\right)  $. Since $P_{1}$ is the projector of $X$
to $\ker L=\ker\left(  1+\Delta^{-1}\right)  $ and $\nabla$ is commutable with
$1+\Delta^{-1}$, so $P_{1}$ is also commutable with $\nabla$. Then the
estimates of $\partial\omega\left(  t\right)  $ follows in the same way as in
the proof of Lemma \ref{lemma-Euler-growth}.
\end{proof}

Similarly, we can estimate the difference of solutions of
(\ref{linearized Euler-projected}) and (\ref{linearized-NS-projected}).

\begin{lemma}
Let $\omega^{\nu},\omega^{0}\ $be the solutions of the projected linearized NS
equation (\ref{eqn-bar-LNS}) and Euler equation (\ref{eqn-linearized Euler})
with the initial data $\omega^{\nu}\left(  0\right)  \in X_{1}$ and
$\omega_{0}\in X_{1}\cap H^{1}.$ Then there exists constants $C_{0}>0\ $such
that
\[
\frac{d}{dt}\left\Vert \omega^{\nu}-\omega^{0}\right\Vert _{X}^{2}\leq
C_{0}\nu\left(  1+t^{2}\right)  \left\Vert \omega^{0}\left(  t\right)
\right\Vert _{H^{1}}^{2},
\]
for $t\in\left(  0,+\infty\right)  $.
\end{lemma}

The proof is the same as that of Lemma \ref{lemma-difference-NS-Euler} by
using the fact that $L|_{X_{1}}>0$ and $\left\langle L\left(  I-P_{1}\right)
\cdot,\cdot\right\rangle =\left\langle L\cdot,\cdot\right\rangle $.

Then by the same proof of Theorem \ref{thm-linearized} i), we can show the
enhanced damping for the solution $\omega_{1}\left(  t\right)  \ $of the
projected equation (\ref{linearized-NS-projected}). More precisely, for any
$\tau>0$ and $\delta>0$, if $\nu$ is small enough, then
\[
\left\Vert \omega_{1}\left(  \frac{\tau}{\nu}\right)  \right\Vert _{L^{2}%
}<\delta\left\Vert \omega_{1}\left(  0\right)  \right\Vert _{L^{2}}.
\]
Since $\omega_{1}\left(  t\right)  =\left(  I-P_{1}\right)  \omega\left(
t\right)  $, this proves Theorem \ref{thm-linearized} ii).

\section{Metastability of nonlinear Navier-Stokes equation}

\label{section-metasta}

In this section, we prove the metastability or enhanced damping for the
nonlinear Navier-Stokes equation (\ref{eqn-NS}) on a torus $\mathbb{T}%
_{\alpha}$ $\left(  \alpha\geq1\right)  $. Let $P_{2}$ to be the projection of
$L^{2}\left(  \mathbb{T}_{\alpha}\right)  \ $to the subspace $W_{2}%
=span\left\{  \cos y,\sin y\right\}  $. For any initial data $\omega\left(
0\right)  \in L^{2}$, let
\[
P_{2}\omega\left(  0\right)  =d_{1}\cos y+d_{2}\sin y=D\sin\left(
y+y_{1}\right)  ,
\]
where $D=\sqrt{d_{1}^{2}+d_{2}^{2}}$ and $y_{1}=\tan^{-1}\left(  d_{1}%
/d_{2}\right)  $. When $\left\Vert \left(  I-P_{2}\right)  \omega\left(
0\right)  \right\Vert _{L^{2}}$ is small, we can equivalently consider the
perturbation near the shear flow $U\left(  y\right)  =D\sin\left(
y+y_{1}\right)  $ with initial data $\omega\left(  0\right)  $ satisfying
$P_{2}\omega\left(  0\right)  =0$, for which the analysis is almost the same
as for the shear flow $U\left(  y\right)  =\sin y$. For simplicity, in the
proof of Theorem \ref{thm-nonlinear}, we only consider the perturbations near
$U\left(  y\right)  =\sin y$ with $P_{2}\omega\left(  0\right)  =0$. As in the
proof of Theorem \ref{thm-linearized} for the linearized NS equation, we will
treat the high and low frequency parts of the non-shear perturbation
separately. In particular, for the low frequency part, we will compare the
solutions of the nonlinear NS equation and the linearized Euler equation, and
then use the RAGE Theorem to control the time average. However, a significant
difference with the linearized NS equation is that the shear and non-shear
parts are strongly coupled for the nonlinear NS\ equation. Therefore, the main
issue is to control the interaction terms. We will consider the equations on
rectangular and square tori separately. On the square torus, the existence of
anomalous modes makes the interactions terms considerably more subtle to handle.

\subsection{The case of rectangular torus}

Consider the nonlinear Navier-Stokes equation near the Kolmogorov flow
\begin{align}
\partial_{t}\omega^{\nu}  &  =\nu\Delta\omega^{\nu}-e^{-\nu t}\left[  \sin
y\partial_{x}\left(  1+\Delta^{-1}\right)  \right]  \omega^{\nu}+U^{\nu}%
\cdot\nabla\omega^{\nu}\label{eqn-NS-nonlinear}\\
&  =L\left(  t\right)  \omega^{\nu}+U^{\nu}\cdot\nabla\omega^{\nu},\nonumber
\end{align}
on $\mathbb{T}_{\alpha}$ $\left(  \alpha>1\right)  $, where $\omega^{\nu
},U^{\nu}$ are the perturbations of vorticity and velocity. We split
$\omega^{\nu},U^{\nu}$ into shear and non-shear components. More precisely, we
write $\omega^{\nu}=\omega_{s}^{\nu}+\omega_{n}^{\nu},\ $ where the shear
part
\[
\omega_{s}^{\nu}\left(  t,y\right)  =\frac{\alpha}{2\pi}\int_{0}^{\frac{2\pi
}{\alpha}}\omega^{\nu}\left(  t,x,y\right)  dx=P_{0}\omega^{\nu},\
\]
and non-shear part
\[
\omega_{n}^{\nu}\left(  t,x,y\right)  =\left(  I-P_{0}\right)  \omega^{\nu
}=P_{\neq0}\omega^{\nu}\in X.
\]
Correspondingly, $U^{\nu}=U_{s}^{\nu}+U_{n}^{\nu}$ and $\psi^{\nu}=\psi
_{s}^{\nu}+\psi_{n}^{\nu}$ where $\psi^{\nu}$ is the stream function. We also
denote
\[
U^{\nu}=\left(  u^{\nu},v^{\nu}\right)  ,\ \ U_{s}^{\nu}=\left(  u_{s}^{\nu
}\left(  t,y\right)  ,0\right)  ,\ U_{n}^{\nu}=\left(  u_{n}^{\nu},v_{n}^{\nu
}\right)  .
\]
Then the equation (\ref{eqn-NS-nonlinear}) can be written as
\begin{equation}
\partial_{t}\omega_{s}^{\nu}=\nu\partial_{yy}\omega_{s}^{\nu}+P_{0}\left(
U_{n}^{\nu}\cdot\nabla\omega_{n}^{\nu}\right)  =\nu\partial_{yy}\omega
_{s}^{\nu}+\partial_{y}P_{0}\left(  v_{n}^{\nu}\omega_{n}^{\nu}\right)  ,
\label{eqn-shear-NS}%
\end{equation}
and
\begin{align}
\partial_{t}\omega_{n}^{\nu}  &  =L\left(  t\right)  \omega_{n}^{\nu}%
+u_{s}^{\nu}\partial_{x}\omega_{n}^{\nu}+v_{n}^{\nu}\partial_{y}\omega
_{s}^{\nu}+P_{\neq0}\left(  U_{n}^{\nu}\cdot\nabla\omega_{n}^{\nu}\right)
\label{eqn-non-shear-NS}\\
&  =L\left(  t\right)  \omega_{n}^{\nu}+u_{s}^{\nu}\partial_{x}\omega_{n}%
^{\nu}+v_{n}^{\nu}\partial_{y}\omega_{s}^{\nu}+\partial_{x}\left(  u_{n}^{\nu
}\omega_{n}^{\nu}\right)  +\partial_{y}P_{\neq0}\left(  v_{n}^{\nu}\omega
_{n}^{\nu}\right)  .\nonumber
\end{align}

First, we show that the dissipation law (\ref{dissipation-law-NS}) also holds
true for solutions of the nonlinear equation (\ref{eqn-NS-nonlinear}).

\begin{lemma}
\label{lemma-dissipation-NS-nonlinear}Let $\omega^{\nu}\left(  t\right)  $ be
a solution of (\ref{eqn-NS-nonlinear}) with the initial data $\omega^{\nu
}\left(  t\right)  \in L^{2}\left(  \mathbb{T}_{\alpha}\right)  $. Then
\begin{equation}
\frac{d}{dt}\int_{\mathbb{T}_{\alpha}}(|\omega^{\nu}|^{2}-|\bigtriangledown
\psi^{\nu}|^{2})dxdy=-2\nu\int_{\mathbb{T}_{\alpha}}(|\bigtriangledown
\omega^{\nu}|^{2}-|\omega^{\nu}|^{2})dxdy, \label{dissipation-law-Nonlinear}%
\end{equation}
for any $t>0$.
\end{lemma}

\begin{proof}
We have
\begin{align*}
\frac{d}{dt}\int_{\mathbb{T}_{\alpha}}(|\omega^{\nu}|^{2}-|\bigtriangledown
\psi^{\nu}|^{2})dxdy  &  =\int_{\mathbb{T}_{\alpha}}\omega_{t}^{\nu}\left(
\omega^{\nu}-\psi^{\nu}\right)  dxdy\\
&  =\left\langle L\left(  t\right)  \omega^{\nu},\left(  \omega^{\nu}%
-\psi^{\nu}\right)  \right\rangle +\int_{\mathbb{T}_{\alpha}}U^{\nu}%
\cdot\nabla\omega^{\nu}\left(  \omega^{\nu}-\psi^{\nu}\right)  dxdy\\
&  =-2\nu\int_{\mathbb{T}_{\alpha}}(|\bigtriangledown\omega^{\nu}|^{2}%
-|\omega^{\nu}|^{2})dxdy+\frac{1}{2}\int_{\mathbb{T}_{\alpha}}U^{\nu}%
\cdot\nabla\frac{1}{2}\omega^{\nu2}dxdy\\
&  \ \ \ \ \ \ \ \ \ \ \ \ \ -\int_{\mathbb{T}_{\alpha}}\left(  U^{\nu}%
\cdot\nabla\psi^{\nu}\right)  \omega^{\nu}dxdy\\
&  =-2\nu\int_{\mathbb{T}_{\alpha}}(|\bigtriangledown\omega^{\nu}|^{2}%
-|\omega^{\nu}|^{2})dxdy.
\end{align*}
In the above, we use the fact that $U^{\nu}\cdot\nabla\psi^{\nu}=0$ and
\[
\left\langle L\left(  t\right)  \omega^{\nu},\left(  \omega^{\nu}-\psi^{\nu
}\right)  \right\rangle =-2\nu\int_{\mathbb{T}_{\alpha}}(|\bigtriangledown
\omega^{\nu}|^{2}-|\omega^{\nu}|^{2})dxdy
\]
as in the proof of Lemma \ref{lemma-dissi-linear}.
\end{proof}

Denote $Y$ to be the space of mean zero functions in $L^{2}\left(
\mathbb{T}_{\alpha}\right)  $, depending only on $y$. Denote $Y_{1}=\left(
I-P_{2}\right)  Y$. Then the operator $L=1+\Delta^{-1}$ is positive on $Y_{1}%
$. There exists $c_{0}>0$ such that
\begin{equation}
\int_{\mathbb{T}_{\alpha}}(|\omega^{\nu}|^{2}-|\partial_{y}\psi^{\nu}%
|^{2})dxdy\geq c_{0}\left\Vert \omega^{\nu}\right\Vert _{L^{2}}^{2},
\label{positivity-L2-y}%
\end{equation}
and
\begin{equation}
\int_{\mathbb{T}_{\alpha}}(|\partial_{y}\omega^{\nu}|^{2}-|\omega^{\nu}%
|^{2})dxdy\geq c_{0}\left\Vert \omega\right\Vert _{H^{1}}^{2},
\label{positivity-H1-y}%
\end{equation}
for all $\omega\in Y_{1}$.

For a solution $\omega^{\nu}=\omega_{s}^{\nu}+\omega_{n}^{\nu}$ of the
nonlinear NS equation (\ref{eqn-NS-nonlinear}), let $\omega_{s}^{\nu}%
=\omega_{s1}^{\nu}+\omega_{s2}^{\nu}$, where $\omega_{s1}^{\nu}=P_{2}%
\omega_{s}^{\nu}$ and $\omega_{s2}^{\nu}\in Y_{1}$. Then
(\ref{dissipation-law-Nonlinear}) implies that
\begin{align}
\ \  &  \frac{d}{dt}\left(  \int_{\mathbb{T}_{\alpha}}(|\omega_{s2}^{\nu}%
|^{2}-|\partial_{y}\psi_{s2}^{\nu}|^{2})dxdy+\left\Vert \omega_{n}^{\nu
}\right\Vert _{X}^{2}\right) \label{dissipation-coupled}\\
&  =-2\nu\left(  \int_{\mathbb{T}_{\alpha}}(|\partial_{y}\omega_{s2}^{\nu
}|^{2}-|\omega_{s2}^{\nu}|^{2})dxdy+\left\Vert \omega_{n}^{\nu}\right\Vert
_{X^{1}}^{2}\right)  .\nonumber
\end{align}
In particular, by the positivity estimates (\ref{positive-L2}),
(\ref{positive-H1}), it follows from above that there exists $C>0$ such that
\begin{equation}
\left\Vert \omega_{s2}^{\nu}\right\Vert _{L^{2}}\left(  t\right)  +\left\Vert
\omega_{n}^{\nu}\right\Vert _{L^{2}}\left(  t\right)  \leq C\left(  \left\Vert
\omega_{s2}^{\nu}\right\Vert _{L^{2}}\left(  0\right)  +\left\Vert \omega
_{n}^{\nu}\right\Vert \left(  0\right)  _{L^{2}}\right)  \leq Cd\nu,
\label{estimate-shear-non-kernel-nonshear}%
\end{equation}
where we assume $\left\Vert \omega^{\nu}\right\Vert _{L^{2}}\left(  0\right)
\leq d\nu$ for a constant $d>0$ to be determined later. To estimate
$\omega_{s1}^{\nu}=a_{1}\cos y+a_{2}\sin y$, we project (\ref{eqn-shear-NS})
to $\left\{  \cos y,\sin y\right\}  $ to get
\begin{align}
\frac{d}{dt}a_{1}  &  =-\nu a_{1}-\int_{\mathbb{T}_{\alpha}}v_{n}^{\nu}%
\omega_{n}^{\nu}\sin y\ dxdy,\label{eqn-a}\\
\frac{d}{dt}a_{2}  &  =-\nu a_{2}+\int_{\mathbb{T}_{\alpha}}v_{n}^{\nu}%
\omega_{n}^{\nu}\cos y\ dxdy.\nonumber
\end{align}
Let
\[
a\left(  t\right)  =\left\Vert \omega_{s1}^{\nu}\right\Vert _{L^{2}}%
=\sqrt{a_{1}^{2}+a_{2}^{2}},
\]
then by (\ref{eqn-a}), we have
\begin{align*}
\frac{da}{dt}  &  \leq-\nu a\left(  t\right)  +\sqrt{\left\vert \int%
_{\mathbb{T}_{\alpha}}v_{n}^{\nu}\omega_{n}^{\nu}\sin y\ dxdy\right\vert
^{2}+\left\vert \int_{\mathbb{T}_{\alpha}}v_{n}^{\nu}\omega_{n}^{\nu}\cos
y\ dxdy\right\vert ^{2}}\\
&  \leq-\nu a\left(  t\right)  +\left\Vert v_{n}^{\nu}\right\Vert \left(
t\right)  _{L^{2}}\left\Vert \omega_{n}^{\nu}\right\Vert \left(  t\right)
_{L^{2}}\\
&  \leq-\nu a\left(  t\right)  +\left(  Cd\nu\right)  ^{2}.
\end{align*}
Therefore
\[
\left\Vert \omega_{s1}^{\nu}\right\Vert _{L^{2}}=a\left(  t\right)  \leq
e^{-\nu t}a\left(  0\right)  +\left(  Cd\nu\right)  ^{2}\int_{0}^{t}%
e^{-\nu\left(  t-s\right)  }ds\leq Cd\nu.
\]
Combined with (\ref{estimate-shear-non-kernel-nonshear}), it follows from
above that
\begin{equation}
\left\Vert \omega_{s}^{\nu}\right\Vert _{L^{2}}\left(  t\right)  \leq
Cd\nu,\ \text{for all }t>0\text{. } \label{estimate-shear-L2}%
\end{equation}

In the dissipation law (\ref{dissipation-coupled}), $\omega_{s2}^{\nu}$ and
$\omega_{n}^{\nu}$ are coupled. In the next lemma, we show that when $d$ is
small, the dissipation law for $\omega_{n}^{\nu}$ can be "separated" from
(\ref{dissipation-coupled}).

\begin{lemma}
\label{lemma-dissipation-non-shear}There exists a constant $d$ depending only
on $\alpha$, such that when $\left\Vert \omega^{\nu}\right\Vert _{L^{2}%
}\left(  0\right)  \leq d\nu$, then
\begin{equation}
\frac{d}{dt}\int_{\mathbb{T}_{\alpha}}\left\Vert \omega_{n}^{\nu}\right\Vert
_{X}^{2}\leq-\nu\left\Vert \omega_{n}^{\nu}\right\Vert _{X^{1}}^{2}.
\label{dissipation-law-square-non-shear}%
\end{equation}

\end{lemma}

\begin{proof}
By (\ref{eqn-shear-NS}), we have
\begin{align*}
&  \ \ \ \ \frac{d}{dt}\int_{\mathbb{T}_{\alpha}}(|\omega_{s2}^{\nu}%
|^{2}-|\partial_{y}\psi_{s2}^{\nu}|^{2})dxdy\\
&  =2\int_{\mathbb{T}_{\alpha}}\partial_{t}\omega_{s2}^{\nu}\left(
\omega_{s2}^{\nu}-\psi_{s2}^{\nu}\right)  dxdy=2\int_{\mathbb{T}_{\alpha}%
}\partial_{t}\omega_{s}^{\nu}\left(  \omega_{s2}^{\nu}-\psi_{s2}^{\nu}\right)
dxdy\\
&  =-2\nu\int_{\mathbb{T}_{\alpha}}(|\partial_{y}\omega_{s2}^{\nu}%
|^{2}-|\omega_{s2}^{\nu}|^{2})dxdy+2\int_{\mathbb{T}_{\alpha}}\partial
_{y}\left(  v_{n}^{\nu}\omega_{n}^{\nu}\right)  \left(  \omega_{s2}^{\nu}%
-\psi_{s2}^{\nu}\right)  dxdy\\
&  \leq-2\nu\int_{\mathbb{T}_{\alpha}}(|\partial_{y}\omega_{s2}^{\nu}%
|^{2}-|\omega_{s2}^{\nu}|^{2})dxdy+C\left\Vert \omega_{n}^{\nu}\right\Vert
_{X^{1}}^{2}\left\Vert \omega_{s2}^{\nu}\right\Vert _{L^{2}}\\
&  \leq-2\nu\int_{\mathbb{T}_{\alpha}}(|\partial_{y}\omega_{s2}^{\nu}%
|^{2}-|\omega_{s2}^{\nu}|^{2})dxdy+Cd\nu\left\Vert \omega_{n}^{\nu}\right\Vert
_{X^{1}}^{2}.
\end{align*}
In the first inequality above, we use the Sobolev embedding and
(\ref{positive-H1}). Combined with (\ref{dissipation-coupled}), this gives
\[
\frac{d}{dt}\int_{\mathbb{T}_{\alpha}}\left\Vert \omega_{n}^{\nu}\right\Vert
_{X}^{2}\leq-\nu\left\Vert \omega_{n}^{\nu}\right\Vert _{X^{1}}^{2}\left(
2-Cd\right)  \leq-\nu\left\Vert \omega_{n}^{\nu}\right\Vert _{X^{1}}^{2},
\]
when $d\leq1/C$.
\end{proof}

To use the RAGE theorem to control the low frequency part of $\omega_{n}^{\nu
}$, we need to estimate the difference of the solutions of the nonlinear NS
equation (\ref{eqn-NS-nonlinear}) and the linearized Euler equation
(\ref{eqn-linearized Euler}).

\begin{lemma}
\label{lemma-difference-NS-non-shear}There exists $d>0$ such that for any
solution $\omega^{\nu}\ $of Navier-Stokes equation (\ref{eqn-NS-nonlinear})
with initial data satisfying $\left\Vert \omega^{\nu}\right\Vert _{L^{2}%
}\left(  0\right)  \leq d\nu$, and any solution $\omega^{0}$ of the linearized
Euler equation (\ref{eqn-linearized Euler}) with initial data in $X\cap H^{1}%
$, we have
\[
\frac{d}{dt}\left\Vert \omega_{n}^{\nu}-\omega^{0}\right\Vert _{X}^{2}\leq
C_{0}\nu\left(  1+t^{2}\right)  \left\Vert \omega^{0}\left(  t\right)
\right\Vert _{H^{1}}^{2},\ \forall t>0,
\]
for some constant $C_{0}>0$. Here, $\omega_{n}^{\nu}=P_{\neq0}\omega^{\nu}$ is
the non-shear part of $\omega^{\nu}$.
\end{lemma}

\begin{proof}
Let $\psi_{n}^{\nu},\psi^{0}$ be the corresponding stream functions. Denote
$\omega=\omega_{n}^{\nu}-\omega^{0}$ and $\psi=\psi_{n}^{\nu}-\psi^{0}$, then
\begin{align*}
\omega_{t}  &  =-e^{-\nu t}\sin y\partial_{x}(\omega-\psi)-(e^{-\nu t}-1)\sin
y\partial_{x}(\omega^{0}-\psi^{0})+\nu\bigtriangleup\omega_{n}^{\nu}\\
&  \ \ \ \ \ \ \ \ \ \ +u_{s}^{\nu}\partial_{x}\omega_{n}^{\nu}+v_{n}^{\nu
}\partial_{y}\omega_{s}^{\nu}+P_{\neq0}\left(  U_{n}^{\nu}\cdot\nabla
\omega_{n}^{\nu}\right)  .
\end{align*}
Thus
\[
\ \ \ \frac{d}{dt}\frac{1}{2}\left\Vert \omega^{\nu}-\omega^{0}\right\Vert
_{X}^{2}=\int_{\mathbb{T}_{\alpha}}\omega_{t}\left(  \omega-\psi\right)  dxdy
\]%
\begin{align*}
&  =\left[  -\int_{\mathbb{T}_{\alpha}}(e^{-\nu t}-1)\sin y\partial_{x}%
(\omega^{0}-\psi^{0})\left(  \omega-\psi\right)  dxdy+\nu\int_{\mathbb{T}%
_{\alpha}}\bigtriangleup\omega_{n}^{\nu}(\omega-\psi)dxdy\right] \\
&  \ \ \ \ \ \ +\int_{\mathbb{T}_{\alpha}}u_{s}^{\nu}\partial_{x}\omega
_{n}^{\nu}\left(  \omega-\psi\right)  dxdy+\int_{\mathbb{T}_{\alpha}}%
v_{n}^{\nu}\partial_{y}\omega_{s}^{\nu}\left(  \omega-\psi\right)
dxdy+\int_{\mathbb{T}_{\alpha}}U_{n}^{\nu}\cdot\nabla\omega_{n}^{\nu}\left(
\omega-\psi\right)  dxdy\\
&  =I+II+III+IV.
\end{align*}
Similar to the proof of Lemma \ref{lemma-difference-NS-Euler}, the first term
can be estimated by%
\[
I\leq\nu\left(  -c_{0}\left\Vert \omega_{n}^{\nu}\right\Vert _{H^{1}}%
^{2}+C\left(  1+t\right)  \left\Vert \omega^{\nu}\right\Vert _{H^{1}%
}\left\Vert \omega^{0}\right\Vert _{H^{1}}\right)  .
\]
For the last term, noticing that as in the proof of Lemma
\ref{lemma-dissipation-NS-nonlinear}
\[
\int_{\mathbb{T}_{\alpha}}U_{n}^{\nu}\cdot\nabla\omega_{n}^{\nu}\left(
\omega_{n}^{\nu}-\psi_{n}^{\nu}\right)  dxdy=0,
\]
we have
\begin{align*}
IV  &  =-\int_{\mathbb{T}_{\alpha}}U_{n}^{\nu}\cdot\nabla\omega_{n}^{\nu
}\left(  \omega^{0}-\psi^{0}\right)  dxdy=\int_{\mathbb{T}_{\alpha}}U_{n}%
^{\nu}\cdot\nabla\left(  \omega^{0}-\psi^{0}\right)  \omega_{n}^{\nu}dxdy\\
&  \leq C\left\Vert U_{n}^{\nu}\right\Vert _{L^{4}}\left\Vert \omega
^{0}\right\Vert _{H^{1}}\left\Vert \omega_{n}^{\nu}\right\Vert _{L^{4}}\leq
C\left\Vert \omega_{n}^{\nu}\right\Vert _{L^{2}}\left\Vert \omega
^{0}\right\Vert _{H^{1}}\left\Vert \omega_{n}^{\nu}\right\Vert _{H^{1}}\\
&  \leq Cd\nu\left\Vert \omega^{0}\right\Vert _{H^{1}}\left\Vert \omega
_{n}^{\nu}\right\Vert _{H^{1}}.
\end{align*}
The second term is estimated by
\begin{align*}
II  &  =\int_{\mathbb{T}_{\alpha}}u_{s}^{\nu}\partial_{x}\omega_{n}^{\nu
}\left(  \psi_{n}^{\nu}-\omega^{0}+\psi^{0}\right)  dxdy\\
&  =-\int_{\mathbb{T}_{\alpha}}u_{s}^{\nu}\omega_{n}^{\nu}\left(  \partial
_{x}\psi_{n}^{\nu}-\partial_{x}\left(  \omega^{0}-\psi^{0}\right)  \right)
dxdy\\
&  \leq\left\Vert u_{s}^{\nu}\right\Vert _{L^{\infty}}\left\Vert \omega
_{n}^{\nu}\right\Vert _{L^{2}}\left\Vert v_{n}^{\nu}\right\Vert _{L^{2}%
}+\left\Vert u_{s}^{\nu}\right\Vert _{L^{\infty}}\left\Vert \omega_{n}^{\nu
}\right\Vert _{L^{2}}\left\Vert \omega^{0}\right\Vert _{H^{1}}\\
&  \leq C\left(  \left\Vert \omega_{s}^{\nu}\right\Vert _{L^{2}}\left\Vert
\omega_{n}^{\nu}\right\Vert _{H^{1}}^{2}+\left\Vert \omega_{s}^{\nu
}\right\Vert _{L^{2}}\left\Vert \omega_{n}^{\nu}\right\Vert _{H^{1}}\left\Vert
\omega^{0}\right\Vert _{H^{1}}\right) \\
&  \leq Cd\nu\left(  \left\Vert \omega_{n}^{\nu}\right\Vert _{H^{1}}%
^{2}+\left\Vert \omega_{n}^{\nu}\right\Vert _{H^{1}}\left\Vert \omega
^{0}\right\Vert _{H^{1}}\right)  ,
\end{align*}
where we use (\ref{estimate-shear-L2}) in the last inequality. Similarly,
\begin{align*}
III  &  =-\int_{\mathbb{T}_{\alpha}}U_{n}^{\nu}\cdot\nabla\left(  \omega
-\psi\right)  \omega_{s}^{\nu}dxdy\\
&  \leq C\left\Vert U_{n}^{\nu}\right\Vert _{L^{\infty}}\left(  \left\Vert
\omega_{n}^{\nu}\right\Vert _{H^{1}}+\left\Vert \omega^{0}\right\Vert _{H^{1}%
}\right)  \left\Vert \omega_{s}^{\nu}\right\Vert _{L^{2}}\\
&  \leq Cd\nu\left(  \left\Vert \omega_{n}^{\nu}\right\Vert _{H^{1}}%
^{2}+\left\Vert \omega^{0}\right\Vert _{H^{1}}\left\Vert \omega_{n}^{\nu
}\right\Vert _{H^{1}}\right)  ,
\end{align*}
where the embedding $\left\Vert U_{n}^{\nu}\right\Vert _{L^{\infty}}\leq
C\left\Vert \omega_{n}^{\nu}\right\Vert _{H^{1}}$ and the bound
(\ref{estimate-shear-L2}) are used in the last inequality above. Combing above
estimates, we get
\begin{align*}
\frac{d}{dt}\frac{1}{2}\left\Vert \omega^{\nu}-\omega^{0}\right\Vert _{X}^{2}
&  \leq\nu\left(  -\left(  c_{0}-Cd\right)  \left\Vert \omega_{n}^{\nu
}\right\Vert _{H^{1}}^{2}+C\left(  1+t\right)  \left\Vert \omega_{n}^{\nu
}\right\Vert _{H^{1}}\left\Vert \omega^{0}\right\Vert _{H^{1}}\right) \\
&  \leq\nu\left(  -\frac{1}{2}c_{0}\left\Vert \omega_{n}^{\nu}\right\Vert
_{H^{1}}^{2}+C\left(  1+t\right)  \left\Vert \omega_{n}^{\nu}\right\Vert
_{H^{1}}\left\Vert \omega^{0}\right\Vert _{H^{1}}\right) \\
&  \leq C\nu\left(  1+t^{2}\right)  \left\Vert \omega^{0}\left(  t\right)
\right\Vert _{H^{1}}^{2},
\end{align*}
by choosing $d\leq\frac{1}{2C}$.
\end{proof}

By using Lemmas \ref{lemma-dissipation-NS-nonlinear} and
\ref{lemma-difference-NS-non-shear}, Theorem \ref{thm-nonlinear} i) follows by
the same arguments as in the proof of Theorem \ref{thm-linearized}.

\subsection{The case of square torus}

Consider the nonlinear Navier-Stokes equation (\ref{eqn-NS-nonlinear}) on the
square torus $\mathbb{T}$, for initial data satisfying $P_{2}\omega\left(
0\right)  =0$. Compared with the rectangular torus, the new difficulty is the
existence of anomalous modes $\left\{  \cos x,\sin x\right\}  $ in the kernel
of $L=1+\Delta^{-1}$. We decompose the vorticity perturbation as
\[
\omega^{\nu}=\omega_{s}^{\nu}+\omega_{n}^{\nu}=\omega_{s1}^{\nu}+\omega
_{s2}^{\nu}+\omega_{n1}^{\nu}+\omega_{n2}^{\nu},
\]
where the shear part $\omega_{s}^{\nu}$ is decomposed as in the rectangular
case with $\omega_{s1}^{\nu}=P_{2}\omega_{s}^{\nu}$ and $\omega_{s2}^{\nu}\in
Y_{1}$, and the non-shear part $\omega_{n}^{\nu}$ is decomposed as
$\omega_{n1}^{\nu}+\omega_{n2}^{\nu}$ with $\omega_{n1}^{\nu}=P_{1}\omega
_{n}^{\nu}$ and $\omega_{n2}^{\nu}\in X_{1}$. Here, we recall that $P_{1}$ is
the projection to the anomalous space $W_{1}$ spanned by $\left\{  \cos x,\sin
x\right\}  $ and $X_{1}$ is the orthogonal complement of $W_{1}$ in $X$.
Correspondingly, the velocity $U^{\nu}\ $and stream function $\psi$ are
decomposed into four parts. Then the nonlinear term can be written as
\begin{align}
U^{\nu}\cdot\nabla\omega^{\nu} &  =\left(  U_{s1}^{\nu}+U_{n1}^{\nu}\right)
\cdot\nabla\left(  \omega_{n2}^{\nu}-\psi_{n2}^{\nu}\right)  +U_{n1}^{\nu
}\cdot\nabla\left(  \omega_{s2}^{\nu}-\psi_{s2}^{\nu}\right)
\label{nonlinear-term-decom-torus}\\
&  \ \ \ \ \ \ \ +U_{s2}^{\nu}\cdot\nabla\omega_{n2}^{\nu}+U_{n2}^{\nu}%
\cdot\nabla\omega_{s2}^{\nu}+U_{n2}^{\nu}\cdot\nabla\omega_{n2}^{\nu
},\nonumber
\end{align}
where we use the observation
\[
U_{s1}^{\nu}\cdot\nabla\omega_{n}^{\nu}+U_{n}^{\nu}\cdot\nabla\omega_{s2}%
^{\nu}=U_{s1}^{\nu}\cdot\nabla\left(  \omega_{n}^{\nu}-\psi_{n}^{\nu}\right)
=U_{s1}^{\nu}\cdot\nabla\left(  \omega_{n2}^{\nu}-\psi_{n2}^{\nu}\right)
\]
and similarly for other terms. The dissipation law (\ref{dissipation-coupled})
becomes
\begin{align}
\ \ \  &  \frac{d}{dt}\left(  \int_{\mathbb{T}}(|\omega_{s2}^{\nu}%
|^{2}-|\partial_{y}\psi_{s2}^{\nu}|^{2})dxdy+\left\Vert \omega_{n2}^{\nu
}\right\Vert _{X}^{2}\right)  \label{dissipation-coupled-torus}\\
&  =-2\nu\left(  \int_{\mathbb{T}}(|\partial_{y}\omega_{s2}^{\nu}|^{2}%
-|\omega_{s2}^{\nu}|^{2})dxdy+\left\Vert \omega_{n2}^{\nu}\right\Vert _{X^{1}%
}^{2}\right)  .\nonumber
\end{align}
By using the positivity of the above functional on $Y_{1}$ and $X_{1}$, this
implies that
\begin{equation}
\left\Vert \omega_{s2}^{\nu}\right\Vert _{L^{2}}\left(  t\right)  +\left\Vert
\omega_{n2}^{\nu}\right\Vert _{L^{2}}\leq C\left(  \left\Vert \omega_{s2}%
^{\nu}\right\Vert _{L^{2}}\left(  0\right)  +\left\Vert \omega_{n2}^{\nu
}\right\Vert \left(  0\right)  _{L^{2}}\right)  \leq Cd\nu
.\label{estimate-non-shear-non-anmous}%
\end{equation}
We estimate $\left\Vert \omega_{s1}^{\nu}\right\Vert _{L^{2}}\left(  t\right)
$ and $\left\Vert \omega_{n1}^{\nu}\right\Vert _{L^{2}}$ below. Denote
\[
\omega_{s1}^{\nu}=a_{1}\cos y+a_{2}\sin y,\ \omega_{n1}^{\nu}=b_{1}\cos
x+b_{2}\sin x,
\]
and
\[
a\left(  t\right)  =\left\Vert \omega_{s1}^{\nu}\right\Vert _{L^{2}}\left(
t\right)  =\sqrt{a_{1}^{2}+a_{2}^{2}},\ b\left(  t\right)  =\left\Vert
\omega_{n1}^{\nu}\right\Vert _{L^{2}}\left(  t\right)  =\sqrt{b_{1}^{2}%
+b_{2}^{2}}.
\]
Since $P_{2}\omega\left(  0\right)  =0$, we have $a\left(  0\right)  =0$. By
(\ref{nonlinear-term-decom-torus}), we have
\begin{align}
\partial_{t}\omega_{s}^{\nu} &  =\nu\partial_{yy}\omega_{s}^{\nu}+P_{0}\left(
U_{n}^{\nu}\cdot\nabla\omega_{n}^{\nu}\right)  \label{eqn-shear-ns-torus}\\
&  =\nu\partial_{yy}\omega_{s}^{\nu}+P_{0}\left(  U_{n1}^{\nu}\cdot
\nabla\left(  \omega_{n2}^{\nu}-\psi_{n2}^{\nu}\right)  +U_{n2}^{\nu}%
\cdot\nabla\omega_{n2}^{\nu}\right)  .\nonumber
\end{align}
Projecting above to $\left\{  \cos y,\sin y\right\}  $ and using
(\ref{estimate-non-shear-non-anmous}), we get
\[
\frac{da_{i}}{dt}\leq-\nu a_{i}\left(  t\right)  +Cd\nu\left(  b\left(
t\right)  +\left\Vert \omega_{n2}^{\nu}\right\Vert _{L^{2}}\right)  ,\ i=1,2,
\]
and thus%
\begin{equation}
\frac{da}{dt}\leq-\nu a\left(  t\right)  +Cd\nu\left(  b\left(  t\right)
+\left\Vert \omega_{n2}^{\nu}\right\Vert _{L^{2}}\right)  .\label{eqn-a-torus}%
\end{equation}
Projecting (\ref{eqn-NS-nonlinear}) to $\left\{  \cos x,\sin x\right\}  $ and
using (\ref{nonlinear-term-decom-torus}) (\ref{estimate-non-shear-non-anmous}%
), we get
\[
\frac{db_{i}}{dt}\leq-\nu b_{i}\left(  t\right)  +Cd\nu\left(  \left(
a\left(  t\right)  +b\left(  t\right)  \right)  +\left\Vert \omega_{n2}^{\nu
}\right\Vert _{L^{2}}\right)  ,\ i=1,2,
\]
and
\begin{equation}
\frac{db}{dt}\leq-\nu b\left(  t\right)  +Cd\nu\left(  \left(  a\left(
t\right)  +b\left(  t\right)  \right)  +\left\Vert \omega_{n2}^{\nu
}\right\Vert _{L^{2}}\right)  .\label{eqn-b}%
\end{equation}
Let $e\left(  t\right)  =a\left(  t\right)  +b\left(  t\right)  $, then the
combination of (\ref{eqn-a-torus}) and (\ref{eqn-b}) yields%
\begin{equation}
\frac{de}{dt}\leq-\nu\left(  1-Cd\right)  e\left(  t\right)  +Cd\nu\left\Vert
\omega_{n2}^{\nu}\right\Vert _{L^{2}}\leq-\frac{1}{2}\nu e\left(  t\right)
+Cd\nu\left\Vert \omega_{n2}^{\nu}\right\Vert _{L^{2}},\label{eqn-e}%
\end{equation}
by choosing $Cd\leq\frac{1}{2}$. Thus by (\ref{estimate-non-shear-non-anmous}%
), we have%
\begin{align}
\left\Vert \omega_{s1}^{\nu}\right\Vert _{L^{2}}\left(  t\right)  +\left\Vert
\omega_{n1}^{\nu}\right\Vert _{L^{2}} &  =e\left(  t\right)
\label{estimate-anomous-modes}\\
&  \leq e^{-\frac{1}{2}\nu t}e\left(  0\right)  +Cd\nu\int_{0}^{t}e^{-\frac
{1}{2}\nu\left(  t-s\right)  }\left\Vert \omega_{n2}^{\nu}\right\Vert _{L^{2}%
}\left(  s\right)  ds\leq Cd\nu.\nonumber
\end{align}
In the following lemma, we separate the dissipation law for $\omega_{n2}^{\nu
}$ from (\ref{dissipation-coupled-torus}).

\begin{lemma}
\label{lemma-dissipation-torus-non-anomous}There exists a constant $d>0$, such
that when $\left\Vert \omega^{\nu}\right\Vert _{L^{2}}\left(  0\right)  \leq
d\nu$, then
\begin{equation}
\frac{d}{dt}\left\Vert \omega_{n2}^{\nu}\right\Vert _{X}^{2}\leq-\nu\left\Vert
\omega_{n2}^{\nu}\right\Vert _{X^{1}}\left(  \left\Vert \omega_{n2}^{\nu
}\right\Vert _{X^{1}}-\left\Vert \omega_{n1}^{\nu}\right\Vert _{L^{2}}\right)
. \label{dissipation-w-n2}%
\end{equation}

\end{lemma}

\begin{proof}
By (\ref{eqn-shear-ns-torus}), we have
\begin{align*}
&  \ \ \ \ \ \frac{d}{dt}\int_{\mathbb{T}}(|\omega_{s2}^{\nu}|^{2}%
-|\partial_{y}\psi_{s2}^{\nu}|^{2})dxdy\\
&  =2\int_{\mathbb{T}}\partial_{t}\omega_{s2}^{\nu}\left(  \omega_{s2}^{\nu
}-\psi_{s2}^{\nu}\right)  dxdy=2\int_{\mathbb{T}}\partial_{t}\omega_{s}^{\nu
}\left(  \omega_{s2}^{\nu}-\psi_{s2}^{\nu}\right)  dxdy\\
&  =-2\nu\int_{\mathbb{T}}(|\partial_{y}\omega_{s2}^{\nu}|^{2}-|\omega
_{s2}^{\nu}|^{2})dxdy\\
&  \ \ \ \ \ \ \ \ \ \ \ \ \ \ \ \ +2\int_{\mathbb{T}_{\alpha}}\left(
U_{n1}^{\nu}\cdot\nabla\left(  \omega_{n2}^{\nu}-\psi_{n2}^{\nu}\right)
+U_{n2}^{\nu}\cdot\nabla\omega_{n2}^{\nu}\right)  \left(  \omega_{s2}^{\nu
}-\psi_{s2}^{\nu}\right)  dxdy\\
&  \leq-2\nu\int_{\mathbb{T}_{\alpha}}(|\partial_{y}\omega_{s2}^{\nu}%
|^{2}-|\omega_{s2}^{\nu}|^{2})dxdy+C\left(  \left\Vert \omega_{n1}^{\nu
}\right\Vert _{L^{2}}\left\Vert \omega_{n2}^{\nu}\right\Vert _{X^{1}%
}+\left\Vert \omega_{n2}^{\nu}\right\Vert _{X^{1}}^{2}\right)  \left\Vert
\omega_{s2}^{\nu}\right\Vert _{L^{2}}\\
&  \leq-2\nu\int_{\mathbb{T}_{\alpha}}(|\partial_{y}\omega_{s2}^{\nu}%
|^{2}-|\omega_{s2}^{\nu}|^{2})dxdy+Cd\nu\left(  \left\Vert \omega_{n1}^{\nu
}\right\Vert _{L^{2}}\left\Vert \omega_{n2}^{\nu}\right\Vert _{X^{1}%
}+\left\Vert \omega_{n2}^{\nu}\right\Vert _{X^{1}}^{2}\right)  .
\end{align*}
By choosing $d$ such that $Cd\leq1$, (\ref{dissipation-w-n2}) follows from
above and (\ref{dissipation-coupled-torus}).
\end{proof}

Compared to the dissipation law (\ref{dissipation-law-square-non-shear}) for
the rectangular torus, from (\ref{dissipation-w-n2}) we cannot even infer that
$\left\Vert \omega_{n2}^{\nu}\right\Vert _{X}^{2}$ is decreasing due to the
interaction of $\omega_{n2}^{\nu}$ and $\omega_{n1}^{\nu}$. The dissipation
law (\ref{dissipation-coupled-torus}) only implies that
\begin{align*}
&  \int_{\mathbb{T}}(|\omega_{s2}^{\nu}|^{2}\left(  t\right)  -|\partial
_{y}\psi_{s2}^{\nu}|^{2}\left(  t\right)  )dxdy+\left\Vert \omega_{n2}^{\nu
}\left(  t\right)  \right\Vert _{X}^{2}\\
&  \leq\int_{\mathbb{T}}(|\omega_{s2}^{\nu}|^{2}\left(  0\right)
-|\partial_{y}\psi_{s2}^{\nu}|^{2}\left(  0\right)  )dxdy+\left\Vert
\omega_{n2}^{\nu}\left(  0\right)  \right\Vert _{X}^{2},
\end{align*}
while the shear term $\omega_{s2}^{\nu}$ cannot be separated from above.
Below, we show that $\left\Vert \omega_{n2}^{\nu}\left(  t\right)  \right\Vert
_{X}$ and $\left\Vert \omega_{n1}^{\nu}\right\Vert _{L^{2}}$ are controlled by
their initial values with a factor $e^{C\nu t}$, which is enough for proving
the enhanced damping.

\begin{lemma}
\label{lemma-growth-torus}There exists $d>0$, such that when $\left\Vert
\omega^{\nu}\right\Vert _{L^{2}}\left(  0\right)  \leq d\nu$, then
\begin{equation}
\left\Vert \omega_{n2}^{\nu}\left(  t\right)  \right\Vert _{X}^{2}+\left\Vert
\omega_{n1}^{\nu}\right\Vert _{L^{2}}\left(  t\right)  ^{2}+\left\Vert
\omega_{s1}^{\nu}\right\Vert _{L^{2}}^{2}\left(  t\right)  \leq e^{C_{1}\nu
t}\left\Vert \omega_{n}^{\nu}\left(  0\right)  \right\Vert _{L^{2}}%
^{2}\text{,} \label{inequality-growth-torus}%
\end{equation}
for some constant $C_{1}>0$.
\end{lemma}

\begin{proof}
It follows from (\ref{dissipation-w-n2}) that
\[
\frac{d}{dt}\left\Vert \omega_{n2}^{\nu}\right\Vert _{X}^{2}\leq\frac{1}{4}%
\nu\left\Vert \omega_{n1}^{\nu}\right\Vert _{L^{2}}^{2}.
\]
From (\ref{eqn-e}), we have
\[
\frac{de^{2}}{dt}\leq-\nu e\left(  t\right)  ^{2}+Cd\nu\left\Vert \omega
_{n2}^{\nu}\right\Vert _{X}e\left(  t\right)  \leq\frac{1}{4}\left(
Cd\right)  ^{2}\nu\left\Vert \omega_{n2}^{\nu}\right\Vert _{X}^{2}\text{. }%
\]
Combining above and denoting $C_{1}=\max\left\{  \frac{1}{4},\frac{1}%
{4}\left(  Cd\right)  ^{2}\right\}  $, then we have
\begin{align*}
\frac{d}{dt}\left(  \left\Vert \omega_{n2}^{\nu}\right\Vert _{X}^{2}+e\left(
t\right)  ^{2}\right)   &  \leq C_{1}\nu\left(  \left\Vert \omega_{n2}^{\nu
}\right\Vert _{X}^{2}+\left\Vert \omega_{n1}^{\nu}\right\Vert _{L^{2}}%
^{2}\right) \\
&  \leq C_{1}\nu\left(  \left\Vert \omega_{n2}^{\nu}\right\Vert _{X}%
^{2}+e\left(  t\right)  ^{2}\right)  .
\end{align*}
Since $\left\Vert \omega_{s1}^{\nu}\right\Vert _{L^{2}}\left(  0\right)  =0$,
we get
\begin{align*}
\left\Vert \omega_{n2}^{\nu}\right\Vert _{X}^{2}+e\left(  t\right)  ^{2}  &
\leq e^{C_{1}\nu t}\left(  \left\Vert \omega_{n2}^{\nu}\right\Vert _{X}%
^{2}\left(  0\right)  +\left\Vert \omega_{n1}^{\nu}\right\Vert _{L^{2}}%
^{2}\left(  0\right)  \right) \\
&  \leq e^{C_{1}\nu t}\left(  \left\Vert \omega_{n2}^{\nu}\right\Vert _{L^{2}%
}^{2}\left(  0\right)  +\left\Vert \omega_{n1}^{\nu}\right\Vert _{L^{2}}%
^{2}\left(  0\right)  \right)  =e^{C_{1}\nu t}\left\Vert \omega_{n}^{\nu
}\left(  0\right)  \right\Vert _{L^{2}}^{2}.
\end{align*}
This finishes the proof of the lemma.
\end{proof}

Our next lemma is to estimate the difference of $\omega_{n2}^{\nu}$ and the
solution of the linearized Euler equation.

\begin{lemma}
\label{lemma-difference-torus}There exists $d>0$ such that for any solution
$\omega^{\nu}\ $of the Navier-Stokes equation (\ref{eqn-non-shear-NS}) on
$\mathbb{T\ }$with initial data satisfying $\left\Vert \omega^{\nu}\right\Vert
_{L^{2}}\left(  0\right)  \leq d\nu$, and any solution $\omega^{0}$ of the
linearized Euler equation (\ref{eqn-linearized Euler}) with initial data in
$X_{1}\cap H^{1}$, we have
\begin{equation}
\frac{d}{dt}\left\Vert \omega_{n2}^{\nu}-\omega^{0}\right\Vert _{X}^{2}\leq
C_{0}\nu\left(  \left(  1+t^{2}\right)  \left\Vert \omega^{0}\left(  t\right)
\right\Vert _{H^{1}}^{2}+\left\Vert \omega_{n1}^{\nu}\right\Vert _{L^{2}}%
^{2}+\left\Vert \omega_{n1}^{\nu}\right\Vert _{L^{2}}\left\Vert \omega
^{0}\right\Vert _{H^{1}}\right)  ,\ \label{estimate-difference-torus}%
\end{equation}
$\forall t>0$, for some constant $C_{0}>0$. Here, $\omega_{n2}^{\nu}=\left(
I-P_{1}\right)  P_{\neq0}\omega^{\nu}$ is the non-shear part of $\omega^{\nu}$
with anomalous modes removed.
\end{lemma}

\begin{proof}
Denote $\omega=\omega_{n2}^{\nu}-\omega^{0}$ and $\psi=\psi_{n2}^{\nu}%
-\psi^{0}$, then
\begin{align*}
\omega_{t}  &  =-e^{-\nu t}\left(  I-P_{1}\right)  \sin y\partial_{x}%
(\omega-\psi)-(e^{-\nu t}-1)\left(  I-P_{1}\right)  \sin y\partial_{x}%
(\omega^{0}-\psi^{0})+\nu\bigtriangleup\omega_{n2}^{\nu}\\
&  \ \ \ \ \ \ \ \ \ \ +\left(  I-P_{1}\right)  P_{\neq0}\left(  \left(
U_{s1}^{\nu}+U_{n1}^{\nu}\right)  \cdot\nabla\left(  \omega_{n2}^{\nu}%
-\psi_{n2}^{\nu}\right)  \right)  +\left(  I-P_{1}\right)  \left(  U_{n1}%
^{\nu}\cdot\nabla\left(  \omega_{s2}^{\nu}-\psi_{s2}^{\nu}\right)  \right) \\
&  \ \ \ \ \ \ \ \ \ \ +\left(  I-P_{1}\right)  \left(  U_{s2}^{\nu}%
\cdot\nabla\omega_{n2}^{\nu}+U_{n2}^{\nu}\cdot\nabla\omega_{s2}^{\nu}\right)
+\left(  I-P_{1}\right)  P_{\neq0}\left(  U_{n2}^{\nu}\cdot\nabla\omega
_{n2}^{\nu}\right)  .
\end{align*}
So we have
\[
\ \ \ \frac{d}{dt}\frac{1}{2}\left\Vert \omega_{n2}^{\nu}-\omega
^{0}\right\Vert _{X}^{2}=\int_{\mathbb{T}_{\alpha}}\omega_{t}\left(
\omega-\psi\right)  dxdy
\]%
\begin{align*}
&  =\left[  -\int_{\mathbb{T}}(e^{-\nu t}-1)\sin y\partial_{x}(\omega^{0}%
-\psi^{0})\left(  \omega-\psi\right)  dxdy+\nu\int_{\mathbb{T}_{\alpha}%
}\bigtriangleup\omega_{n2}^{\nu}(\omega-\psi)dxdy\right] \\
&  \ \ \ \ \ \ +\left[  \int_{\mathbb{T}}\left(  U_{s2}^{\nu}\cdot\nabla
\omega_{n2}^{\nu}+U_{n2}^{\nu}\cdot\nabla\omega_{s2}^{\nu}\right)  \left(
\omega-\psi\right)  dxdy+\int_{\mathbb{T}}U_{n2}^{\nu}\cdot\nabla\omega
_{n2}^{\nu}\left(  \omega-\psi\right)  dxdy\right] \\
&  \ \ \ \ \ \ -\left[  \int_{\mathbb{T}}\left(  U_{s1}^{\nu}+U_{n1}^{\nu
}\right)  \cdot\nabla\left(  \omega_{n2}^{\nu}-\psi_{n2}^{\nu}\right)  \left(
\omega^{0}-\psi^{0}\right)  dxdy\right] \\
&  \ \ \ \ \ \ +\int_{\mathbb{T}}U_{n1}^{\nu}\cdot\nabla\left(  \omega
_{s2}^{\nu}-\psi_{s2}^{\nu}\right)  \left(  \omega-\psi\right)  dxdy\\
&  =I+II+III+IV.
\end{align*}

For the first three terms, as in the proof of Lemma
\ref{lemma-difference-NS-non-shear}, we get
\[
I+II+III\leq-\frac{1}{2}c_{0}\nu\left\Vert \omega_{n2}^{\nu}\right\Vert
_{H^{1}}^{2}+C\nu\left(  1+t\right)  \left\Vert \omega_{n2}^{\nu}\right\Vert
_{H^{1}}\left\Vert \omega^{0}\right\Vert _{H^{1}},
\]
by choosing $d>0$ small enough. The last term is estimated by%
\begin{align*}
IV  &  =-\int_{\mathbb{T}}U_{n1}^{\nu}\cdot\nabla\left[  \left(  \omega
_{n2}^{\nu}-\psi_{n2}^{\nu}\right)  -\left(  \omega^{0}-\psi^{0}\right)
\right]  \left(  \omega_{s2}^{\nu}-\psi_{s2}^{\nu}\right)  dxdy\\
&  \leq C\nu\left\Vert \omega_{n1}^{\nu}\right\Vert _{L^{2}}\left(  \left\Vert
\omega_{n2}^{\nu}\right\Vert _{H^{1}}+\left\Vert \omega^{0}\right\Vert
_{H^{1}}\right)  .
\end{align*}
Combining above, we have
\[
\frac{d}{dt}\left\Vert \omega_{n2}^{\nu}-\omega^{0}\right\Vert _{X}^{2}\leq
C_{0}\nu\left(  \left(  1+t^{2}\right)  \left\Vert \omega^{0}\right\Vert
_{H^{1}}^{2}+\left\Vert \omega_{n1}^{\nu}\right\Vert _{L^{2}}^{2}+\left\Vert
\omega_{n1}^{\nu}\right\Vert _{L^{2}}\left\Vert \omega^{0}\right\Vert _{H^{1}%
}\right)  .
\]

\end{proof}

We are now ready to prove Theorem \ref{thm-nonlinear} ii).

\begin{proof}
[Proof of Theorem \ref{thm-nonlinear} ii)]For any fixed $\delta,\tau>0$,
suppose that (\ref{enhanced damping torus}) is not true. Then
\begin{equation}
\left\Vert \omega_{n2}^{\nu}\left(  t\right)  \right\Vert _{L^{2}}\geq
\delta\left\Vert \omega_{n}^{\nu}\left(  0\right)  \right\Vert _{L^{2}%
},\ \forall\ \ 0\leq t\leq\frac{\tau}{\nu}, \label{assumption-contra-torus}%
\end{equation}
where $\omega^{\nu}\left(  t\right)  $ is the solution of
(\ref{eqn-NS-nonlinear}) with initial data $\omega^{\nu}\left(  0\right)  \in
L^{2}\left(  \mathbb{T}\right)  $ satisfying $P_{2}\omega^{\nu}\left(
0\right)  =0$ and $\left\Vert \omega^{\nu}\left(  0\right)  \right\Vert
_{L^{2}}\leq d\nu$. Here, $d$ is a constant chosen such that Lemmas
\ref{lemma-dissipation-torus-non-anomous}, \ref{lemma-growth-torus} and
\ref{lemma-difference-torus} hold true. Then by (\ref{inequality-growth-torus}%
) and (\ref{assumption-contra-torus}), for any $t,t^{\prime}\in\left(
0,\frac{\tau}{\nu}\right)  ,$ we have
\begin{align}
\left\Vert \omega_{n1}^{\nu}\left(  t\right)  \right\Vert _{L^{2}}  &  \leq
e^{\frac{1}{2}C_{1}\tau}\left\Vert \omega_{n}^{\nu}\left(  0\right)
\right\Vert _{L^{2}}^{2}\leq\frac{1}{\delta}e^{\frac{1}{2}C_{1}\tau}\left\Vert
\omega_{n2}^{\nu}\left(  t^{\prime}\right)  \right\Vert _{L^{2}}%
\label{inequa-w-n1}\\
&  \leq\frac{c_{0}}{\delta}e^{\frac{1}{2}C_{1}\tau}\left\Vert \omega_{n2}%
^{\nu}\left(  t^{\prime}\right)  \right\Vert _{X}.\nonumber
\end{align}
Choose $\lambda_{N}$ to be big enough such that%
\[
\exp\left(  -\left(  \frac{\lambda_{N}}{2}-\sqrt{\frac{\lambda_{N}}{2}}%
\frac{c_{0}}{\delta}e^{\frac{1}{2}C_{1}\tau}\right)  \tau\right)  <c_{0}%
\delta^{2},
\]
where $C_{1},c_{0}$ are the constants in (\ref{inequality-growth-torus}) and
(\ref{positivity-L2-X1}) respectively. Suppose that $\omega_{n2}^{\nu}\left(
t\right)  \in R\left(  I-P_{N}\right)  $, that is,
\[
\left\Vert \omega_{n2}^{\nu}\left(  t\right)  \right\Vert _{X^{1}}^{2}%
>\lambda_{N}\left\Vert \omega_{n2}^{\nu}\left(  t\right)  \right\Vert _{X}%
^{2},
\]
is true for $t$ in some interval $\left(  a,b\right)  \subset\left(
0,\tau/\nu\right)  $. Then by (\ref{dissipation-w-n2}) and (\ref{inequa-w-n1}%
), we have
\begin{equation}
\left\Vert \omega^{\nu}\left(  b\right)  \right\Vert _{X}^{2}\leq\exp\left(
-\nu\sqrt{\lambda_{N}}\left(  \sqrt{\lambda_{N}}-\frac{c_{0}}{\delta}%
e^{\frac{1}{2}C_{1}\tau}\right)  \left(  b-a\right)  \right)  \left\Vert
\omega^{\nu}\left(  a\right)  \right\Vert _{X}^{2}.
\label{decay-high-frequency}%
\end{equation}
Denote $t_{1}=T_{c}(N,\frac{1}{10},)$ to be such that the RAGE Lemma
(\ref{rage-torus}) is true for $\kappa=1/10$ and all $T\geq T_{c}$. For any
$t_{0}\in\left(  0,\tau/\nu\right)  $ satisfying
\[
\left\Vert \omega_{n2}^{\nu}\left(  t_{0}\right)  \right\Vert _{X^{1}}^{2}%
\leq\lambda_{N}\left\Vert \omega_{n2}^{\nu}\left(  t_{0}\right)  \right\Vert
_{X}^{2},
\]
let $\omega^{0}\left(  t\right)  $ $\left(  t\in\left[  t_{0},t_{0}%
+t_{1}\right]  \right)  \ $be the solution of (\ref{eqn-linearized Euler})
with $\omega^{0}\left(  t_{0}\right)  =\omega_{n2}^{\nu}\left(  t_{0}\right)
$. By (\ref{estimate-difference-torus}), (\ref{inequa-w-n1}) and Lemma
\ref{lemma-growth-torus}, we have
\begin{align}
\left\Vert \omega_{n2}^{\nu}\left(  t_{0}+t\right)  -\omega^{0}\left(
t_{0}+t\right)  \right\Vert _{X}^{2}  &  \leq C_{2}\nu{\Large \{}\left(
1+t^{5}\right)  \left\Vert \omega_{n2}^{\nu}\left(  t_{0}\right)  \right\Vert
_{X^{1}}^{2}+\left(  \frac{c_{0}}{\delta}\right)  ^{2}e^{C_{1}\tau}\left\Vert
\omega_{n2}^{\nu}\left(  t_{0}\right)  \right\Vert _{X}^{2}%
\label{estimate-difference-int-torus}\\
&  \ \ \ +\frac{c_{0}}{\delta}e^{\frac{1}{2}C_{1}\tau}\left(  1+t^{2}\right)
\left\Vert \omega_{n2}^{\nu}\left(  t_{0}\right)  \right\Vert _{X}\left\Vert
\omega_{n2}^{\nu}\left(  t_{0}\right)  \right\Vert _{X^{1}}{\LARGE \}}%
,\ \nonumber
\end{align}
for any $t\in\left[  0,t_{1}\right]  $ and some constant $C_{2}>0$ independent
of $\tau,\nu$. Let $\nu\left(  \delta,\tau\right)  $ be such that
\[
C_{2}\nu\left(  \delta,\tau\right)  \left\{  \left(  1+t_{1}^{5}\right)
\lambda_{N}+\left(  \frac{c_{0}}{\delta}\right)  ^{2}e^{C_{1}\tau}+\frac
{c_{0}}{\delta}e^{\frac{1}{2}C_{1}\tau}\left(  1+t_{1}^{2}\right)
\sqrt{\lambda_{N}}\right\}  <\frac{1}{10}.
\]
By (\ref{estimate-difference-int-torus}), when $0<\nu<\nu\left(  \delta
,\tau\right)  $, we have
\[
\left\Vert \omega_{n2}^{\nu}\left(  t\right)  -\omega^{0}\left(  t\right)
\right\Vert _{X}^{2}\leq\frac{1}{10}\left\Vert \omega_{n2}^{\nu}\left(
t_{0}\right)  \right\Vert _{X}^{2},\ \forall\ t\in\left[  t_{0},t_{0}%
+t_{1}\right]  .
\]
Then as in the proof of Theorem \ref{thm-linearized}, by using the RAGE
theorem for $\omega^{0}\left(  t\right)  $, we obtain%
\[
\int_{t_{0}}^{t_{0}+t_{1}}\left\Vert \omega_{n2}^{\nu}\left(  t\right)
\right\Vert _{X^{1}}^{2}dt\geq\frac{\lambda_{N}t_{1}}{2}\left\Vert \omega
_{n2}^{\nu}\left(  t_{0}\right)  \right\Vert _{X}^{2}.
\]
Thus by (\ref{dissipation-w-n2}) and (\ref{inequa-w-n1}), we have
\begin{align}
\left\Vert \omega_{n2}^{\nu}\left(  t_{0}+t_{1}\right)  \right\Vert _{X}^{2}
&  \leq\left\Vert \omega_{n2}^{\nu}\left(  t_{0}\right)  \right\Vert _{X}%
^{2}-\nu\int_{t_{0}}^{t_{0}+t_{1}}\left\Vert \omega_{n2}^{\nu}\right\Vert
_{X^{1}}\left(  \left\Vert \omega_{n2}^{\nu}\right\Vert _{X^{1}}-\left\Vert
\omega_{n1}^{\nu}\right\Vert _{L^{2}}\right)  dt\label{decay-low-frequency}\\
&  \leq\left\Vert \omega_{n2}^{\nu}\left(  t_{0}\right)  \right\Vert _{X}%
^{2}-\nu\int_{t_{0}}^{t_{0}+t_{1}}\left\Vert \omega_{n2}^{\nu}\right\Vert
_{X^{1}}^{2}dt\nonumber\\
&  \ \ \ \ \ \ \ \ \ +\nu\left(  \int_{t_{0}}^{t_{0}+t_{1}}\left\Vert
\omega_{n2}^{\nu}\right\Vert _{X^{1}}^{2}dt\right)  ^{\frac{1}{2}}\sqrt{t_{1}%
}\frac{c_{0}}{\delta}e^{\frac{1}{2}C_{1}\tau}\left\Vert \omega_{n2}^{\nu
}\left(  t_{0}\right)  \right\Vert _{X}\nonumber\\
&  \leq\left\Vert \omega_{n2}^{\nu}\left(  t_{0}\right)  \right\Vert _{X}%
^{2}\left(  1-\frac{\lambda_{N}}{2}+\sqrt{\frac{\lambda_{N}}{2}}\frac{c_{0}%
}{\delta}e^{\frac{1}{2}C_{1}\tau}\right)  \nu t_{1}.\nonumber\\
&  \leq\left\Vert \omega_{n2}^{\nu}\left(  t_{0}\right)  \right\Vert _{X}%
^{2}\exp\left(  -\left(  \frac{\lambda_{N}}{2}-\sqrt{\frac{\lambda_{N}}{2}%
}\frac{c_{0}}{\delta}e^{\frac{1}{2}C_{1}\tau}\right)  \nu t_{1}\right)
.\nonumber
\end{align}
Splitting the interval $\left[  0,\frac{\tau}{\nu}\right]  $ into a union of
intervals such that either (\ref{decay-high-frequency}) or
(\ref{decay-low-frequency}) is true, then we have
\begin{align*}
\left\Vert \omega_{n2}^{\nu}\left(  \frac{\tau}{\nu}\right)  \right\Vert
_{X}^{2}  &  \leq\exp\left(  -\left(  \frac{\lambda_{N}}{2}-\sqrt
{\frac{\lambda_{N}}{2}}\frac{c_{0}}{\delta}e^{\frac{1}{2}C_{1}\tau}\right)
\tau\right)  \left\Vert \omega_{n2}^{\nu}\left(  0\right)  \right\Vert
_{X}^{2}\\
&  <c_{0}\delta^{2}\left\Vert \omega_{n2}^{\nu}\left(  0\right)  \right\Vert
_{X}^{2}\text{.}%
\end{align*}
This implies that
\[
\left\Vert \omega_{n2}^{\nu}\left(  \frac{\tau}{\nu}\right)  \right\Vert
_{L^{2}}<\delta\left\Vert \omega_{n2}^{\nu}\left(  0\right)  \right\Vert
_{L^{2}}\leq\delta\left\Vert \omega_{n}^{\nu}\left(  0\right)  \right\Vert
_{L^{2}}\text{,}%
\]
which is in contradiction to the assumption (\ref{assumption-contra-torus}).
This finishes the proof of Theorem \ref{thm-nonlinear} ii).
\end{proof}

\begin{remark}
\label{remark-repeat-decay}By (\ref{estimate-non-shear-non-anmous}) and
(\ref{estimate-anomous-modes}), we have the following Liapunov stability
result
\begin{equation}
\left\Vert \omega\left(  t\right)  \right\Vert _{L^{2}}\leq C\left\Vert
\left(  I-P_{2}\right)  \omega\left(  0\right)  \right\Vert _{L^{2}},\ \forall
t>0, \label{Liapunov-stability}%
\end{equation}
for some constant $C>0$ and any solution $\omega\left(  t\right)  $ of the NS
equation (\ref{eqn-NS-nonlinear}). Thus for initial data $\omega\left(
0\right)  $ satisfying
\begin{equation}
\left\Vert \left(  I-P_{2}\right)  \omega\left(  0\right)  \right\Vert
_{L^{2}}\leq\frac{1}{C}d\nu, \label{condition-initial-further}%
\end{equation}
we can repeatedly use Theorem \ref{thm-nonlinear} ii) to get the rapid decay
of the non-shear part with anomalous modes removed, before the dissipation
term takes over. The same remark applies to the rectangular torus case to get
the rapid damping of the non-shear part.
\end{remark}

\begin{remark}
In Theorem \ref{thm-nonlinear} ii), the non-shear part removing the anomalous
modes is reduced to a factor $\delta$ of the initial norm of the whole
non-shear part. This is different from the result
(\ref{enhanced damping torus-linear}) for the linearized NS equation on a
square torus, where the anomalous modes can be separated. The nonlinear
coupling due to the anomalous modes can be seen from the term $U_{n1}^{\nu
}\cdot\nabla\left(  \omega_{s2}^{\nu}-\psi_{s2}^{\nu}\right)  $ in
(\ref{nonlinear-term-decom-torus}), which reflects the nonlinear interaction
of the anomalous modes and the shear part. For the rectangular torus, there is
no such interaction term and the nonlinear enhanced damping result is
consistent with that for the linearized NS\ equation.
\end{remark}

\subsection{Further issues and dipole states}

We comment on some further issues. First, it would be very interesting to
enlarge the metastability basin from $O\left(  \nu\right)  $ in Theorem
\ref{thm-nonlinear} to be $O\left(  \nu^{\alpha}\right)  $ $\left(
0<\alpha<1\right)  $ or independent of $\nu$ if possible. Also, it is
desirable to improve the decay time scale from $O\left(  \tau/\nu\right)  $ to
$O\left(  1/\sqrt{\nu}\right)  $ as given in \cite{beck-wayne} for the
approximated linearized equation (\ref{LNS-appro}). This might require us to
work on initial data of higher regularity. We note that the time scale
$O\left(  1/\sqrt{\nu}\right)  $ in \cite{beck-wayne} was obtained for initial
data in $H^{1}$.

Numerical simulations (\cite{bouchet-simmonet09}) suggest that on the
rectangular torus the bar states (i.e. Kolmogorov flows) are usually observed.
However, on the square torus (\cite{yin-et-final state}), the dipole states of
the form $\omega_{0}=\cos x+\cos y$ or $\sin x+\sin y\ $appear more often than
the bar states. These dipole states are also maximum entropy solutions of the
2D Euler equation, and hence likely candidates for relevant quasi-stationary
states by the statistical approaches of 2D turbulence (e.g.
\cite{robert-survey}). The dipole states represent nonparallel flows with
saddle points on the stream lines and therefore are more difficult to study.
Consider the dipole with $\omega_{0}=\cos x+\cos y$, then the quasi-stationary
Navier-Stokes solution is $\omega^{\nu}\left(  t,x,y\right)  =e^{-\nu
t}\left(  \cos x+\cos y\right)  $. The linearized NS equation around it
becomes
\begin{equation}
\partial_{t}\omega=\nu\Delta\omega+e^{-\nu t}\left[  \left(  \sin
y\partial_{x}-\sin x\partial_{y}\right)  \left(  1+\Delta^{-1}\right)
\right]  \omega. \label{linearized-NS-dipole}%
\end{equation}
There are some similarities with the linearized equation (\ref{eqn-bar-LNS})
near bar states. First, the same dissipation law (\ref{dissipation-law-NS})
holds true for (\ref{linearized-NS-dipole}). The linearized Euler operator is
of the Hamiltonian form
\begin{equation}
\left(  \sin y\partial_{x}-\sin x\partial_{y}\right)  \left(  1+\Delta
^{-1}\right)  =JL,\ \label{defn-JL-dipole}%
\end{equation}
with
\[
J=\sin y\partial_{x}-\sin x\partial_{y},\ \ L=1+\Delta^{-1}.
\]
Consider the energy space $X\ $to be the set of $L^{2}$ functions with zero
mean. Define $P$ to be the projection of $L^{2}$ to $\ker J$. It was shown in
\cite{lin-cmp-04} that for any $\phi\in L^{2}$,
\begin{equation}
P\phi\ |_{\gamma_{i}\left(  c\right)  }=\frac{\oint_{\gamma_{i}\left(
c\right)  }\frac{\phi\left(  x,y\right)  }{\left\vert \nabla\psi
_{0}\right\vert }dl}{\oint_{\gamma_{i}\left(  c\right)  }\frac{1}{\left\vert
\nabla\psi_{0}\right\vert }dl}, \label{formula-projection}%
\end{equation}
where $c$ is in the range of $\psi_{0}=\cos x+\cos y$ and $\gamma_{i}\left(
c\right)  $ is a branch of $\left\{  \psi_{0}=c\right\}  $. Define the
operator $A,\ A_{0}:H^{2}\cap X\rightarrow X$ by
\[
A=-\Delta-1+P,\ \ \ A_{0}=-\Delta-1.
\]
We note that $A\geq A_{0}\geq0$ and
\[
\ker A_{0}=\ker L=\left\{  \cos x,\sin x,\cos y,\sin y\right\}  .
\]
Therefore $\ker A\subset\ker A_{0}$ and by Proposition 2.8 and Lemma 11.3 of
\cite{lin-zeng-hamiltonian}, we have the decomposition
\begin{equation}
L^{2}=\ker\left(  JL\right)  +R\left(  J\right)  . \label{decom-L2-dipole}%
\end{equation}
Here, $\ker\left(  JL\right)  \cap R\left(  J\right)  \subset\ker L$ and
$\ker\left(  JL\right)  ,\ R\left(  J\right)  $ are both invariant under $JL$.
The space $\ker\left(  JL\right)  $ corresponds to the steady solution space
of the linearized Euler equation $\partial_{t}\omega=JL\omega$. Different from
the case of bar states where $\ker\left(  JL\right)  $ is the space of shear
flows, for the dipole states the steady space $\ker\left(  JL\right)  $ has a
more complicated structure. We can restrict the Euler semigroup $e^{tJL}$ on
the invariant subspace $R\left(  J\right)  $. Denote $P_{3}$ to be the
orthogonal projection of $L^{2}\left(  \mathbb{T}\right)  \ $to
\[
\ker L=\left\{  \cos x,\sin x,\cos y,\sin y\right\}  .
\]
We have the following RAGE type result for $JL$ on $R\left(  J\right)  $.

\begin{proposition}
\label{prop-rage-dipole}Suppose that the operator $JL$ defined in
(\ref{defn-JL-dipole}) has no nonzero purely imaginary eigenvalues. Let $B$ be
any compact operator in $L^{2}\left(  \mathbb{T}\right)  $. Then for any
$\omega\left(  0\right)  \in R\left(  J\right)  $, we have
\begin{equation}
\frac{1}{T}\int_{0}^{T}\left\Vert B\left(  I-P_{3}\right)  e^{itL}%
\omega\left(  0\right)  \right\Vert _{L^{2}}^{2}dt\rightarrow0,\ \text{when
}T\rightarrow\infty. \label{RAGE-dipole}%
\end{equation}

\end{proposition}

\begin{proof}
The proof is similar to Lemma \ref{lemma-continuous-A-torus} for bar states.
We only sketch it. For any solution $\omega\left(  t\right)  $ of the equation
$\partial_{t}\omega=JL\omega$ with $\omega\left(  0\right)  \in R\left(
J\right)  $, define $\omega_{1}\left(  t\right)  =\left(  I-P_{3}\right)
\omega\left(  t\right)  $ and let $X_{1}=\left(  I-P_{3}\right)  R\left(
J\right)  $. Then $\omega_{1}\left(  t\right)  $ satisfies the equation
$\partial_{t}\omega_{1}=\left(  I-P_{3}\right)  JL\omega$ on the space $X_{1}%
$. Since $L|_{X_{1}}>0$, the operator $\left(  I-P_{3}\right)  JL$ is
anti-selfadjoint on $\left(  X_{1},\left\langle L\cdot,\cdot\right\rangle
\right)  $. Our assumption on the spectrum of $JL$ implies that $\left(
I-P_{3}\right)  JL$ has no nonzero purely imaginary eigenvalues. To show that
the operator $\left(  I-P_{3}\right)  JL$ has purely continuous spectrum on
$X_{1}$, it remains to show that $0$ is not an embedded eigenvalue of $\left(
I-P_{3}\right)  JL$. Suppose otherwise, there exists $0\neq\omega\in X_{1}$
such that $\left(  I-P_{3}\right)  JL\omega=0$. Let $\psi_{1}=L\omega$, then
$J\psi_{1}=P_{3}JL\omega\in\ker L$. Denote
\begin{equation}
J\psi_{1}=a_{1}\cos x+a_{2}\cos y+b_{1}\sin x+b_{2}\sin y. \label{eqn-phi-1}%
\end{equation}
Since $\cos y+\cos x\in\ker J\perp J\psi_{1},$ we have $a_{1}+a_{2}=0$. Let
$a_{1}=-a_{2}=a$. It is easy to see that for any $c\neq0$ in the range of
$\psi_{0}=\cos x+\cos y$, each of the two branches of $\left\{  \psi
_{0}=c\right\}  $ is symmetric to $x,y$ in the sense that both $\left(
x,y\right)  $ and $\left(  y,x\right)  $ are on the branch. Since any function
$\phi\ $in $\ker J$ takes the form (\ref{formula-projection}), we conclude
that $\cos x-\cos y\perp\ker J$, which implies that $\cos x-\cos y\in R\left(
J\right)  $. So there exists a double periodic function $\psi_{2}$ such that
$J\psi_{2}=\cos x-\cos y$. By noting that $Jx=\sin y$ and $Jy=-\sin x$, it
follows from (\ref{eqn-phi-1}) that $J\left(  \psi_{1}-a\psi_{2}+b_{1}%
y-b_{2}x\right)  =0$. Then by (\ref{formula-projection}), the function
$\psi_{1}-a\psi_{2}+b_{1}y-b_{2}x$ must take constant on each branch of the
level set $\left\{  \psi_{0}=c\right\}  $. Since $\psi_{0}$ is double periodic
in $x,y$, this implies that $\psi_{1}-a\psi_{2}+b_{1}y-b_{2}x$ is also double
periodic. This contradiction shows that $0$ is not an embedded eigenvalue of
$\left(  I-P_{3}\right)  JL\omega$. Then (\ref{RAGE-dipole}) follows from the
standard RAGE theorem.
\end{proof}

Even with the dissipation law (\ref{dissipation-law-NS}) and above RAGE
theorem, there are still significant differences with the bar states to get
the linear enhanced damping for dipoles, besides the issue of proving the
non-existence of imaginary eigenvalues of $JL$. The most important difference
is that the decomposition (\ref{decom-L2-dipole}) is no longer invariant when
the viscosity is added on. It is under investigation to find a subspace of
initial data such that the enhanced damping is true for dipoles.

\section{Linear inviscid damping of shear flows}

\label{section-linear damping}

Consider a shear flow $u_{0}=\left(  U\left(  y\right)  ,0\right)  $ in a
channel $\left\{  y_{1}\leq y\leq y_{2}\right\}  $ or on a torus. The
linearized Euler equation can be written as
\begin{equation}
\omega_{t}+U\left(  y\right)  \partial_{x}\omega+U^{\prime\prime}\left(
y\right)  \partial_{x}\psi=0\text{, } \label{eqn-linearized-euler-shear}%
\end{equation}
where $\omega$ and $\psi=\left(  -\Delta\right)  ^{-1}\omega$ are the
vorticity and stream functions respectively.

\subsection{Stable case}

We consider two classes of stable shear flows.

Class 1: $U^{\prime\prime}\neq0$, that is, $U$ has no inflection points. This
case is restricted to a channel, since such flows can not exist on a torus. By
the classical Rayleigh inflection point theorem, $\left(  U\left(  y\right)
,0\right)  $ is linearly stable. Suppose $U^{\prime\prime}>0$, choose a
constant $U_{s}<\min U$. Then in the frame $\left(  x-U_{s}t,y,t\right)  $,
the equation (\ref{eqn-linearized-euler-shear}) becomes
\begin{equation}
\omega_{t}+\left(  U\left(  y\right)  -U_{s}\right)  \partial_{x}%
\omega+U^{\prime\prime}\left(  y\right)  \partial_{x}\psi=0,
\label{linearized-Euler-shear-frame}%
\end{equation}
Define
\[
K_{1}\left(  y\right)  =\frac{U^{\prime\prime}\left(  y\right)  }{U\left(
y\right)  -U_{s}}>0.
\]
Let the $x$ period be $2\pi/\alpha$ for any $\alpha>0$. Define the non-shear
space on the periodic channel $S_{2\pi/\alpha}\times\left[  y_{1}%
,y_{2}\right]  $ by
\[
X=\left\{  \omega=\sum_{k\in\mathbf{Z},\ k\neq0}e^{ik\alpha x}\omega
_{k}\left(  y\right)  ,\ \Vert\omega\Vert_{X}^{2}=\Vert\frac{1}{\sqrt
{K_{1}\left(  y\right)  }}\omega\Vert_{L^{2}}^{2}<\infty\right\}  .
\]
Clearly, $X\subset L^{2}\ $and $L^{2}=X$ if $\min K_{1}>0$. The equation
(\ref{linearized-Euler-shear-frame}) can be written in a Hamiltonian form
\[
\omega_{t}=-U^{\prime\prime}\left(  y\right)  \partial_{x}\left(  \frac
{\omega}{K_{1}\left(  y\right)  }+\psi\right)  =JL\omega,
\]
where
\[
J=-U^{\prime\prime}\left(  y\right)  \partial_{x}:X^{\ast}\rightarrow
X,\ \ \ \ L=\frac{1}{K_{1}\left(  y\right)  }+\left(  -\Delta\right)
^{-1}:X\rightarrow X^{\ast},
\]
are anti-selfadjoint and self-adjoint respectively. Moreover, $L$ is uniformly
positive on $X$ and thus $JL$ is anti-selfadjoint in the equivalent inner
product $\left\langle L\cdot,\cdot\right\rangle $. Since $X\subset L^{2}$, by
Lemma \ref{lemma-inflection-value}, $JL$ has no purely imaginary eigenvalues
in $X\ $and the entire spectrum of $JL$ is continuous.

Case 2: Assume that there exists $U_{s}$ in the range of $U$ such that
\begin{equation}
K_{2}\left(  y\right)  =-\frac{U^{\prime\prime}}{U-U_{s}}>0
\label{assumption-K2}%
\end{equation}
is bounded. We call these flows to be in class $\mathcal{K}^{+}$, as used in
\cite{lin-shear}. The assumption (\ref{assumption-K2}) implies that $U_{s}$ is
the only inflection value of $U$. Examples include $U\left(  y\right)  =\sin
y,\ \tanh y$, and more generally any $U\left(  y\right)  $ satisfying the ODE
$U^{\prime\prime}=g\left(  U\right)  $ with a decreasing $g$. Then
(\ref{assumption-K2}) is satisfied with $U_{s}=g^{-1}\left(  0\right)  $. Let
the $x$ period be $2\pi/\alpha$. We can consider the class $\mathcal{K}^{+}$
flows in a periodic channel $S_{2\pi/\alpha}\times\left[  y_{1},y_{2}\right]
$ or on a torus $S_{2\pi/\alpha}\times S_{y_{2}-y_{1}}$. The linearized Euler
equation (\ref{linearized-Euler-shear-frame}) with $U_{s}$ being the
inflection value can be written in the Hamiltonian form%
\begin{equation}
\omega_{t}=U^{\prime\prime}\left(  y\right)  \partial_{x}\left(  \frac{\omega
}{K_{2}\left(  y\right)  }-\psi\right)  =JL\omega,
\label{linearized-operator-shear-K}%
\end{equation}
where
\[
J=U^{\prime\prime}\left(  y\right)  \partial_{x},\ \ \ \ L=\frac{1}%
{K_{2}\left(  y\right)  }-\left(  -\Delta\right)  ^{-1}.
\]
Define the non-shear space of vorticity
\begin{equation}
X=\left\{  \omega=\sum_{k\in\mathbf{Z},\ k\neq0}e^{ik\alpha x}\omega
_{k}\left(  y\right)  ,\ \Vert\omega\Vert_{X}^{2}=\Vert\frac{1}{\sqrt
{K_{2}\left(  y\right)  }}\omega\Vert_{L^{2}}^{2}<\infty\right\}  .
\label{defn-X-K2}%
\end{equation}
Again, $X\subset L^{2}\ $and $L^{2}=X$ if $\min K_{2}>0$.\ Denote
$n^{-}\left(  L\right)  $ $\left(  n^{0}\left(  L\right)  \right)  \ $to be
the number of negative (zero) directions of $L$ on $X$. Define the operator%
\begin{equation}
A_{0}=-\Delta-K_{2}\left(  y\right)  :H^{2}\rightarrow L^{2} \label{defn-A0}%
\end{equation}
in the channel $S_{2\pi/\alpha}\times\left[  y_{1},y_{2}\right]  $ or on the
torus $S_{2\pi/\alpha}\times S_{y_{2}-y_{1}}$ and
\begin{equation}
L_{0}=-\frac{d^{2}}{dy^{2}}-K_{2}\left(  y\right)  :H^{2}\left(  y_{1}%
,y_{2}\right)  \rightarrow L^{2}\left(  y_{1},y_{2}\right)  , \label{defn-L0}%
\end{equation}
with the Dirichlet boundary conditions for the channel and periodic boundary
conditions for the torus. Then by Lemma 11.3 in \cite{lin-zeng-hamiltonian},
we have
\[
n^{0}\left(  L\right)  =n^{0}\left(  A_{0}\right)  =2\sum_{k\geq1}n^{0}\left(
L_{0}+k^{2}\alpha^{2}\right)  ,
\]
and
\[
n^{-}\left(  L\right)  =n^{-}\left(  A_{0}\right)  =2\sum_{k\geq1}n^{-}\left(
L_{0}+k^{2}\alpha^{2}\right)  .
\]
If $n^{-}\left(  L_{0}\right)  \neq0$, let $-\alpha_{\max}^{2}$ be the
smallest eigenvalue of $L_{0}$ and $\phi_{0}$ be an eigenfunction. When
$L_{0}\geq0$, let $\alpha_{\max}=0$. Then by the above relations, $L$ is
positive when $\alpha>\alpha_{\max}$. Again, by Lemma
\ref{lemma-inflection-value}, the spectrum of $JL$ is purely continuous in
$X$. When $\alpha=\alpha_{\max}$, we have $n^{-}\left(  L\right)  =0$ and
\[
\ker L=\left\{  e^{\pm i\alpha x}\omega_{0}\left(  y\right)  \right\}
,\ \omega_{0}\left(  y\right)  =\left(  -\frac{d^{2}}{dy^{2}}+\alpha
^{2}\right)  \phi_{0}.
\]
Let $P_{1}$ be the projection of $X$ to $\ker L$ and $X_{1}=\left(
I-P_{1}\right)  X$. Then $L|_{X_{1}}>0$ and $A_{1}=\left(  I-P_{1}\right)  JL$
is anti-selfadjoint on $\left(  X_{1},\left\langle L\cdot,\cdot\right\rangle
\right)  $.

\begin{lemma}
\label{lemma-continuous-A1-stable}$A_{1}$ has purely continuous spectrum on
$X_{1}$.
\end{lemma}

\begin{proof}
By Lemma 3.5 of \cite{lin-shear}, $\phi_{0}\neq0$ at at least one of the
points in the set $\left\{  U=U_{s}\right\}  $. By using this fact, the rest
of the proof follows that of Lemma \ref{lemma-continuous-A-torus}, so we skip
the details.
\end{proof}

By the above spectral properties, the following is a direct consequence of the
RAGE Theorem.

\begin{theorem}
\label{thm-rage-general-shear}If i) $U^{\prime\prime}\neq0$ or ii) $U$ is in
class $\mathcal{K}^{+}\ $and $\alpha>\alpha_{\max}$, then for any compact
operator $B\ $on $X$, we have
\begin{equation}
\frac{1}{T}\int_{0}^{T}\left\Vert Be^{itJL}\omega\right\Vert _{X}%
^{2}dt\rightarrow0,\ \text{when }T\rightarrow\infty, \label{decay-B-average}%
\end{equation}
for any $\omega\in X$. If $U$ is in class $\mathcal{K}^{+}$ and $\alpha
=\alpha_{\max}$, then for any compact operator $B\ $on $X$, we have
\[
\frac{1}{T}\int_{0}^{T}\left\Vert B\left(  I-P_{1}\right)  e^{itJL}%
\omega\right\Vert _{X}^{2}dt\rightarrow0,\ \text{when }T\rightarrow\infty,
\]
for any $\omega\in X$.
\end{theorem}

By choosing
\[
B\omega=\nabla^{\perp}\left(  -\Delta\right)  ^{-1}\omega=u,
\]
that is, the mapping from vorticity to velocity, we get

\begin{corollary}
\label{cor-decay-general-shear}i) If a) $U^{\prime\prime}\neq0$ or b) $U$ is
in class $\mathcal{K}^{+}$ and $\alpha>\alpha_{\max}$, then
\[
\frac{1}{T}\int_{0}^{T}\left\Vert u\left(  t\right)  \right\Vert _{L^{2}}%
^{2}dt\rightarrow0,\ \text{when }T\rightarrow\infty,
\]
for any solution $\omega\left(  t\right)  \ $of
(\ref{eqn-linearized-euler-shear}) with $\omega\left(  0\right)  \in X$.

ii) If $U$ is in class $\mathcal{K}^{+}$ and $\alpha=\alpha_{\max}$, then
\[
\frac{1}{T}\int_{0}^{T}\left\Vert u_{1}\left(  t\right)  \right\Vert _{L^{2}%
}^{2}dt\rightarrow0,\ \text{when }T\rightarrow\infty,
\]
where $u_{1}\left(  t\right)  $ is the velocity corresponding to the vorticity
$\left(  I-P_{1}\right)  \omega\left(  t\right)  $ with $\omega\left(
0\right)  \in X$.
\end{corollary}

\begin{remark}
\label{rmk-regularity-measure}More information on the decay of $\left\Vert
u\left(  t\right)  \right\Vert _{L^{2}}$ could be obtained by studying the
regularity of the spectral measure of the anti-selfadjoint operator $JL$ on
$\left(  X,\left\langle L\cdot,\cdot\right\rangle \right)  $. In particular,
if the spectral measure of $JL$ is absolutely continuous (i.e. absence of
singular continuous spectrum), then when $t\rightarrow\infty,\ \omega\left(
t\right)  \rightarrow0$ weakly in $X$. \ As a result, $\left\Vert u\left(
t\right)  \right\Vert _{L^{2}}\rightarrow0$ when $t\rightarrow\infty$.
\end{remark}

\begin{remark}
\label{remark-dual-cascade}Let $B=P_{N}$ in (\ref{decay-B-average}), i.e., the
projection operator to the first $N$ Fourier modes (in $x$), then
\begin{equation}
\frac{1}{T}\int_{0}^{T}\left\Vert P_{N}\omega\left(  t\right)  \right\Vert
_{L^{2}}^{2}dt\rightarrow0,\text{when }T\rightarrow\infty.
\label{enstrophy-cascade}%
\end{equation}
This shows that in the time averaged sense, the low frequency part of $\omega$
tends to zero. As can be seen in the proof of Theorems \ref{thm-linearized}
and \ref{thm-nonlinear}, this property plays an important role in the proof of
metastability of Kolmogorov flows. In the fluid literature (see e.g.
\cite{tableling} \cite{bouchett-reports-09}), a dual cascade was known for 2D
turbulence that energy moves to low frequency end and the enstrophy (i.e.
$\left\Vert \omega\right\Vert _{L^{2}}^{2}$) moves to the high frequency end.
The result (\ref{enstrophy-cascade}) can be seen as a justification of such
physical intuition in a weak sense.
\end{remark}

\begin{remark}
\label{rmk-energy-casimir}The two classes of shear flows considered above are
linearly stable in the $L^{2}$ norm of vorticity (assuming $\inf\ K_{i}>0$),
in the Liapunov sense. This follows directly from the conservation of
$\left\langle L\omega,\omega\right\rangle $ for the linearized Euler equation
(\ref{linearized-Euler-shear-frame}) and the positivity of $L|_{X}$. Moreover,
these two classes seem to exhaust all the possible shear flows for which
nonlinear stability could be proved by the energy-Casimir method initiated by
Arnold (\cite{arnold-stability-1} \cite{arnold-stability-2}) in 1960s. We
briefly discuss it below and refer to \cite{marchiro-pulvirenti-book} for more
discussions on energy-Casimir method for 2D Euler equations. The
energy-Casimir functional is of the form
\[
H\left(  \omega\right)  =\int\left(  G\left(  \omega\right)  +\frac{1}%
{2}\left\vert \nabla\psi\right\vert ^{2}\right)  \ dxdy,
\]
which is invariant for the nonlinear Euler equation. Suppose $\psi
_{0}=F\left(  \omega_{0}\right)  $, where $\psi_{0}=\int\left(  U-U_{s}%
\right)  dy$ and $\omega_{0}=-U^{\prime\prime}$. Choose $G$ such that
$G^{\prime}\left(  \omega_{0}\right)  =-F\left(  \omega_{0}\right)  $, then
$H^{\prime}\left(  \omega_{0}\right)  =0$ and the second order variation is
given by
\[
\left\langle H^{\prime\prime}\left(  \omega\right)  \delta\omega,\delta
\omega\right\rangle =\frac{1}{2}\int\frac{\left(  \delta\omega\right)  ^{2}%
}{K_{1}\left(  y\right)  }+\frac{1}{2}\left\vert \nabla\delta\psi\right\vert
^{2}=\frac{1}{2}\left\langle L\delta\omega,\delta\omega\right\rangle
\]
for $U\left(  y\right)  $ in class 1 and
\[
\left\langle H^{\prime\prime}\left(  \omega\right)  \delta\omega,\delta
\omega\right\rangle =\frac{1}{2}\int-\frac{\left(  \delta\omega\right)  ^{2}%
}{K_{2}\left(  y\right)  }+\frac{1}{2}\left\vert \nabla\delta\psi\right\vert
^{2}=-\frac{1}{2}\left\langle L\delta\omega,\delta\omega\right\rangle
\]
for $U\left(  y\right)  $ in class 2. In the above, we use the relation
\[
G^{\prime\prime}\left(  \omega_{0}\right)  =-F^{\prime}\left(  \omega
_{0}\right)  =\frac{U-U_{s}}{U^{\prime\prime}}.
\]
Thus $\left\langle H^{\prime\prime}\left(  \omega\right)  \delta\omega
,\delta\omega\right\rangle $ on $X\ $is positive definite for class 1 and
negative definite for class 2 when $\alpha>\alpha_{\max}$. Then nonlinear
stability (in $L^{2}$ vorticity) could be proved by properly handling the
higher order terms. However, if $\frac{U-U_{s}}{U^{\prime\prime}}%
\ $(equivalently $K_{1},K_{2}$) changes sign, then the quadratic form
$\left\langle H^{\prime\prime}\left(  \omega\right)  \delta\omega,\delta
\omega\right\rangle $ is highly indefinite and the energy-Casimir method does
not work. Despite their above restrictions, the steady flows whose stability
can be studied by energy-Casimir method do appear often as observable coherent
states in 2D turbulence.
\end{remark}

\subsection{Unstable case}

The shear flows $U\left(  y\right)  \ $in class $\mathcal{K}^{+}\ $are proved
to be linearly unstable when the horizontal wave number $\alpha<\alpha_{\max}%
$, see Theorem 1.2 in \cite{lin-shear}. In this subsection, we will prove the
inviscid damping on the center space which is complementary to the stable and
unstable subspaces. The proof of linear instability in \cite{lin-shear} is by
studying the possible neutral limiting modes and the bifurcation of unstable
modes near them. By using the Hamiltonian structure of
(\ref{linearized-Euler-shear-frame}) and the instability index formula in
\cite{lin-zeng-hamiltonian}, we can recover this linear instability criterion.
Moreover, we get more detailed information about the number of unstable modes
which is important to study the inviscid damping on the center space.

\begin{proposition}
\label{prop-index-JL}Consider $U\left(  y\right)  $ in class $\mathcal{K}%
^{+}\ $and $\alpha>0$, where $2\pi/\alpha$ is the $x$ period. Let $k_{u}$ be
the total algebraic multiplicities of unstable eigenvalues of the operator
$JL\ $defined in (\ref{linearized-operator-shear-K}). Then $k_{u}=n^{-}\left(
A_{0}\right)  $, where $A_{0}$ is defined in (\ref{defn-A0}).
\end{proposition}

\begin{proof}
It is easy to see that the unstable eigenfunctions of $JL$ are in the space
$X$ defined in (\ref{defn-X-K2}). Since on the energy space $X$, $n^{-}\left(
L\right)  =n^{-}\left(  A_{0}\right)  <\infty$, we can use Theorem 2.3 in
\cite{lin-zeng-hamiltonian} to get the index formula
\begin{equation}
k_{r}+2k_{c}+2k_{i}^{\leq0}+k_{0}^{\leq0}=n^{-}\left(  L\right)  .
\label{index-formula}%
\end{equation}
Here, $k_{r}$ and $k_{c}$ are the algebraic multiplicities of unstable
eigenvalues of $JL$ lying on the positive axis and the first quadrant
respectively, $k_{i}^{\leq0}$ is the number of non-positive directions of $L$
restricted to the generalized eigenspace of imaginary eigenvalues on
$i\mathbf{R}^{+}$, and $k_{0}^{\leq0}$ is the number of non-positive
directions of $L$ on $E_{0}/\ker L$ where $E_{0}$ is the generalized zero
eigenspace of $JL$. Since $JL$ has no nonzero imaginary eigenvalue, so
$k_{i}^{\leq0}=0$. As in the proof of Lemmas \ref{lemma-continuous-A-torus}
and \ref{lemma-continuous-A1-stable}, it can be shown that $E_{0}=\ker L$ and
thus $k_{0}^{\leq0}=0$. Therefore (\ref{index-formula}) implies that
\[
k_{u}=k_{r}+2k_{c}=n^{-}\left(  L\right)  =n^{-}\left(  A_{0}\right)  .\text{
}%
\]

\end{proof}

The space $X$ has an invariant decomposition $X=\cup_{l\in\mathbf{Z},\ l\neq
0}X^{l}$, where
\[
\ \ X^{l}=\left\{  e^{i\alpha lx}\omega_{l}\left(  y\right)  ,\ \omega_{l}\in
L_{\frac{1}{K_{2}\left(  y\right)  }}^{2}\left(  y_{1},y_{2}\right)
.\right\}  .
\]
On the subspace $X^{l}$, the operator $JL$ is reduced to an ODE operator
$J_{l}L_{l}$ acting on the weighted space $L_{\frac{1}{K_{2}\left(  y\right)
}}^{2}\left(  y_{1},y_{2}\right)  ,$ where
\begin{equation}
J_{l}=U^{\prime\prime}\left(  y\right)  i\alpha l,\ \ \ \ L_{l}=\frac{1}%
{K_{2}\left(  y\right)  }-\left(  -\frac{d^{2}}{dy^{2}}+\alpha^{2}%
l^{2}\right)  ^{-1}. \label{defn-J-L-l}%
\end{equation}
We have a similar counting formula for unstable eigenvalues of $J_{l}L_{l}$.

\begin{proposition}
\label{prop-index-J-L-l}Consider $U\left(  y\right)  $ in class $\mathcal{K}%
^{+}\ $and $\alpha>\alpha_{\max}$. Let $k_{u,l}$ be the total algebraic
multiplicities of the unstable eigenvalues of the operator $J_{l}%
L_{l}\ \left(  0\neq l\in\mathbf{Z}\right)  \ $defined in (\ref{defn-J-L-l}).
Then
\begin{equation}
k_{u,l}=n^{-}\left(  L_{0}+l^{2}\alpha^{2}\right)  , \label{index-formula-l}%
\end{equation}
where $L_{0}$ is defined in (\ref{defn-L0}).
\end{proposition}

\begin{proof}
Since $J_{l}$ is not a real operator, we can not directly apply the index
Theorem 2.3 in \cite{lin-zeng-hamiltonian}. Define the space
\begin{align*}
Y_{l}  &  =\ X^{l}\oplus X^{-l}\\
&  =\left\{  \cos\left(  \alpha lx\right)  \omega_{1}\left(  y\right)
+\sin\left(  \alpha lx\right)  \omega_{2}\left(  y\right)  ,\ \omega
_{1},\omega_{2}\in L_{\frac{1}{K_{2}\left(  y\right)  }}^{2}\left(
y_{1},y_{2}\right)  \right\}  ,
\end{align*}
which is isomorphic to the space $Y=\left(  L_{\frac{1}{K_{2}\left(  y\right)
}}^{2}\left(  y_{1},y_{2}\right)  \right)  ^{2}$. For any
\[
\omega=\cos\left(  \alpha lx\right)  \omega_{1}\left(  y\right)  +\sin\left(
\alpha lx\right)  \omega_{2}\left(  y\right)  \in Y_{l},
\]
we have
\[
JL\omega=\left(  \cos\left(  \alpha lx\right)  ,\sin\left(  \alpha lx\right)
\right)  J^{l}L^{l}\left(
\begin{array}
[c]{c}%
\omega_{1}\\
\omega_{2}%
\end{array}
\right)  ,
\]
where%
\[
J^{l}=\left(
\begin{array}
[c]{cc}%
0 & \alpha lU^{\prime\prime}\\
-\alpha lU^{\prime\prime} & 0
\end{array}
\right)  ,\ \ L^{l}=\left(
\begin{array}
[c]{cc}%
L_{l} & 0\\
0 & L_{l}%
\end{array}
\right)  .
\]
In the above, the operator $L_{l}$ is defined in (\ref{defn-J-L-l}). Thus to
study $JL$ on $Y_{l}$, it is equivalent to study $J^{l}L^{l}$ on $Y$. Let
$k_{u}^{l}\ $be the total algebraic multiplicities of the unstable eigenvalues
of the operator $J^{l}L^{l}$. Then by Theorem 2.3 in
\cite{lin-zeng-hamiltonian} and the same proof of Proposition
\ref{prop-index-J-L-l}, we have
\[
k_{u}^{l}\ =n^{-}\left(  L^{l}\right)  =2n^{-}\left(  L_{l}\right)
=2n^{-}\left(  L_{0}+l^{2}\alpha^{2}\right)  .
\]
Since the spectra of $J_{l}L_{l}$ is complex conjugate of that of
$J_{-l}L_{-l}$, so
\[
k_{u}^{l}=k_{u,l}+k_{u,-l}=2k_{u,l}%
\]
and this finishes the proof of the proposition.
\end{proof}

\begin{remark}
Let $\lambda=i\alpha lc$ be an eigenvalue of $J_{l}L_{l}$ and $J_{l}%
L_{l}\omega=\lambda\omega$ for some $0\neq\omega\in X^{l}$. Then the stream
function
\[
\psi\left(  y\right)  =\left(  -\frac{d^{2}}{dy^{2}}+\alpha^{2}l^{2}\right)
^{-1}\omega\left(  y\right)
\]
satisfies the classical Rayleigh equation
\begin{equation}
\left(  -\frac{d^{2}}{dy^{2}}+\alpha^{2}l^{2}+\frac{U^{\prime\prime}}%
{U-c}\right)  \psi=0 \label{rayleigh-howard}%
\end{equation}
with Dirichlet or periodic boundary conditions. It was shown in
\cite{howard-number} that for $U\left(  y\right)  \ $in class $\mathcal{K}%
^{+}$, the total number (i.e. geometric multiplicities) of unstable
eigenvalues (i.e. $\operatorname{Im}c>0$) can not exceed $n^{-}\left(
L_{0}+l^{2}\alpha^{2}\right)  $. In \cite{lin-shear}, it was shown that
$k_{u,l}\geq1$ when $n^{-}\left(  L_{0}+l^{2}\alpha^{2}\right)  \neq0$. The
precise index formula (\ref{index-formula-l}) not only gives an improvement
over previous results, but also is important below for understanding the
dynamics on the center space.
\end{remark}

Denote $E^{s}\left(  E^{u}\right)  \subset X$ to be the stable (unstable)
eigenspace of $JL$, then
\[
\dim E^{s}=\dim E^{u}=k_{u}=n^{-}\left(  L\right)  \text{. }%
\]
Moreover, by Corollary 6.1 in \cite{lin-zeng-hamiltonian}, $L|_{E^{s}\oplus
E^{u}}$ is non-degenerate and
\begin{equation}
n^{-}\left(  E^{s}\oplus E^{u}\right)  =\dim E^{u}=n^{-}\left(  L\right)  .
\label{negative-index-unstable}%
\end{equation}
Define the center space $E^{c}$ to be the orthogonal (in the inner product
$\left\langle L\cdot,\cdot\right\rangle $) complement of $E^{s}\oplus E^{u}$
in $X$, that is,
\begin{equation}
E^{c}=\left\{  \omega\in X\ |\ \left\langle L\omega,\omega_{1}\right\rangle
=0,\ \forall\omega_{1}\in E^{s}\oplus E^{u}\right\}  . \label{defn-E-c}%
\end{equation}
Then we have

\begin{lemma}
\label{lemma-positive-center}The decomposition $X=E^{s}\oplus E^{u}\oplus
E^{c}$ is invariant under $JL$. Moreover, $n^{-}\left(  L|_{E^{c}}\right)  =0$
and as a consequence $L|_{E^{c}/\ker L}>0$.
\end{lemma}

\begin{proof}
The invariance of the decomposition follows from the invariance of
$\left\langle L\cdot,\cdot\right\rangle $ under $JL$. By
(\ref{negative-index-unstable}), we have
\[
n^{-}\left(  L|_{E^{c}}\right)  =n^{-}\left(  L\right)  -n^{-}\left(
E^{s}\oplus E^{u}\right)  =0,
\]
and thus $L|_{E^{c}/\ker L}>0$.
\end{proof}

Since $E^{c}$ is invariant under $JL$, we can restrict the linearized Euler
equation on $E^{c}$. The linear inviscid damping still holds true for initial
data in $E^{c}$. Denote $P_{1}$ to be the projection of $X$ to $\ker L$. By
the same proof of Theorem \ref{thm-rage-general-shear} and Corollary
\ref{cor-decay-general-shear}, we have the following.

\begin{theorem}
\label{thm-damping-center}Suppose $U\left(  y\right)  $ is in class
$\mathcal{K}^{+}$ and $\alpha<\alpha_{\max}$.

i) If $\ker L=\left\{  0\right\}  $, then
\[
\frac{1}{T}\int_{0}^{T}\left\Vert u\left(  t\right)  \right\Vert _{L^{2}}%
^{2}dt\rightarrow0,\ \text{when }T\rightarrow\infty,
\]
for any solution $\omega\left(  t\right)  \ $of
(\ref{eqn-linearized-euler-shear}) with $\omega\left(  0\right)  \in E^{c}$.
Here, $E^{c}$ is the center space defined in (\ref{defn-E-c}).

ii) If $\ker L\neq\left\{  0\right\}  $, then
\[
\frac{1}{T}\int_{0}^{T}\left\Vert u_{1}\left(  t\right)  \right\Vert _{L^{2}%
}^{2}dt\rightarrow0,\ \text{when }T\rightarrow\infty,
\]
where $u_{1}\left(  t\right)  $ is the velocity corresponding to the vorticity
$\left(  I-P_{1}\right)  \omega\left(  t\right)  $ with $\omega\left(
0\right)  \in E^{c}$.
\end{theorem}

\begin{remark}
Above theorem suggests that the dynamics of solutions of the linearized Euler
equation on the center space $E^{c}$ is similar to the stable case.

The invariant decomposition $X=E^{s}\oplus E^{u}\oplus E^{c}$ can be used to
prove the exponential trichotomy of the semigroup $e^{tJL}$. We refer to
Theorem 2.2 in \cite{lin-zeng-hamiltonian} for the precise statement. The next
natural step is to construct invariant manifolds (stable, unstable and center)
for the nonlinear Euler equation, which will give a complete description of
the local dynamics near $u_{0}=\left(  U\left(  y\right)  ,0\right)  $. The
stable and unstable manifolds near any unstable shear flow were constructed in
\cite{lin-zeng-Euler-mfld}. The construction of center manifold is under
investigation. Once constructed, on such center manifold, the positivity of
$L|_{E^{c}}$ (Lemma \ref{lemma-positive-center}) could be used to prove
nonlinear stability of solutions.
\end{remark}

\section{Appendix}

\begin{lemma}
\label{lemma-inflection-value}Let $U\left(  y\right)  \in C^{2}\left(
y_{1},y_{2}\right)  $, where $-\infty<y_{1}<y_{2}<\infty$. Consider the
Rayleigh equation
\begin{equation}
\left(  -\frac{d^{2}}{dy^{2}}+\alpha^{2}+\frac{U^{\prime\prime}}{U-c}\right)
\psi=0, \label{eqn-Rayleigh}%
\end{equation}
with the periodic boundary condition%
\[
\psi\left(  y_{1}\right)  =\psi\left(  y_{2}\right)  ,\ \psi^{\prime}\left(
y_{1}\right)  =\psi^{\prime}\left(  y_{2}\right)  ,
\]
or the Dirichlet boundary condition
\[
\psi\left(  y_{1}\right)  =\psi\left(  y_{2}\right)  =0.
\]
If (\ref{eqn-Rayleigh}) has a neutral solution with $\alpha>0,\ c\in
\mathbf{R}$ and $\psi\in H^{2}\left(  y_{1},y_{2}\right)  $, then $c$ must be
an inflection value of $U$.
\end{lemma}

\begin{proof}
First, we show that $c$ must be in the range of $U\left(  y\right)  $. Suppose
otherwise $c>\max U$ or $c<\min U$. Assume $c>\max U$. For the Dirichlet
boundary condition, since $U-c<0$ in $\left[  y_{1},y_{2}\right]  $, by the
proof of Lemma 3.5 of \cite{lin-shear}, $\psi\equiv0\ $in $\left[  y_{1}%
,y_{2}\right]  $, which is a contradiction. For the periodic boundary
condition, (\ref{eqn-Rayleigh}) implies that the operator $L_{0}=-\frac{d^{2}%
}{dy^{2}}-\frac{U^{\prime\prime}}{U-c}$ has a negative eigenvalue $-\alpha
^{2}$. Let $\lambda_{0}\leq-\alpha^{2}<0$ be the smallest eigenvalue of
$L_{0}$, then the corresponding eigenfunction $\phi$ can be taken to be
positive. The equation $L_{0}\phi=\lambda_{0}\phi$\ can be written as
\[
\left(  \left(  U-c\right)  \phi^{\prime}-U^{\prime}\phi\right)  ^{\prime
}=-\lambda_{0}\left(  U-c\right)  \phi.
\]
Integrating above from $y_{1}$ to $y_{2}$ and using the periodic boundary
condition, we have
\[
\int_{y_{1}}^{y_{2}}\left(  U-c\right)  \phi dy=0,
\]
which is a contradiction again. Therefore $c$ must be in the range of $U$.

Let $z_{1}<z_{2}<\cdots<z_{k_{c}}$ $\left(  k_{c}\geq1\right)  \ $be the zeros
of $U\left(  y\right)  -c$ in $\left[  y_{1},y_{2}\right]  $. We claim that
there exists $1\leq k\leq k_{c}$ such that $\psi\left(  z_{k}\right)  \neq0$.
For the Dirichlet boundary condition, this follows by Lemma 3.5 of
\cite{lin-shear}. For the periodic boundary condition, suppose otherwise
$\psi\left(  z_{k}\right)  =0$ for all $k=1,2,\cdots,k_{c}$. Let $z_{k_{c}%
+1}=z_{1}+y_{2}-y_{1}$ which is the translation of $z_{1}$ by one period. Then
$U-c$ takes the same sign on each interval $\left[  z_{i},z_{i+1}\right]  $,
$i=1,2,\cdots,k_{c}$, and $\psi=0$ at the end points. By the proof of Lemma
3.5 of \cite{lin-shear}, it follows that $\psi\equiv0$ in all the intervals
$\left[  z_{i},z_{i+1}\right]  $. Thus $\psi\equiv0\ $in $\left[  y_{1}%
,y_{2}\right]  $, a contradiction. Let $1\leq k_{0}\leq k_{c}$ be such that
$\psi\left(  z_{k_{0}}\right)  \neq0$. Then we must have $U^{\prime\prime
}\left(  z_{k_{0}}\right)  =0$, that is, $c=U\left(  z_{k_{0}}\right)  $ is an
inflection value. Suppose otherwise, then $U^{\prime\prime}\left(  z_{k_{0}%
}\right)  \neq0$ and thus $\frac{U^{\prime\prime}}{U-c}\psi$ is not in
$L_{loc}^{2}$ near $z_{k_{0}}$, which is in contradiction to the Rayleigh
equation (\ref{eqn-Rayleigh}) and the assumption that $\psi\in H^{2}$. This
finishes the proof of the Lemma.
\end{proof}

\begin{remark}
The above lemma shows that for general shear flows $U\left(  y\right)  $, any
$H^{2}\ $neutral mode must have its phase speed $c$ to be one of the
inflection values. This fact is used in section \ref{section-linear damping}
to exclude embedded eigenvalues and obtain the instability index formula
(\ref{index-formula-l}) and the positivity of $L|_{E^{c}}$. In
\cite{lin-shear} \cite{lin-shear-note}, it was shown for a class of shear
flows (called class $\mathcal{F}$ in \cite{lin-shear}) that any neutral
limiting mode (i.e. the limit of a sequence of unstable modes) must be in
$H^{2}$ and therefore the phase speed must be inflection values. These neutral
limiting modes are important for finding linear stability/instability criteria
since they give the transition points for stability/instability.

The flows $U\left(  y\right)  \ $in class $\mathcal{F}$ (see \cite{lin-shear}
for definition) include any monotone flow, class $\mathcal{K}^{+}$ flows, and
more generally, any $U\left(  y\right)  $ satisfying an ODE $U^{\prime\prime
}=k\left(  y\right)  g(U)$ for some $k>0$ and any $g$. However, for shear
flows not in class $\mathcal{F}$, the limiting neutral modes might be singular
(i.e. not in $H^{2}$). Such singular neutral modes might have their phase
speeds $c$ other than the inflection values.
\end{remark}

\begin{center}
{\Large Acknowledgement}
\end{center}

Zhiwu Lin is supported in part by NSF grants DMS-1411803 and DMS-1715201.

\end{document}